\documentclass[12pt]{amsart}

\usepackage{amssymb}

\usepackage{enumitem}

\usepackage{graphicx}

\makeatletter
\@namedef{subjclassname@2020}{%
  \textup{2020} Mathematics Subject Classification}
\makeatother

\usepackage[T1]{fontenc}
\newtheorem{theorem}{Theorem}[section]
\newtheorem{corollary}[theorem]{Corollary}
\newtheorem{lemma}[theorem]{Lemma}
\newtheorem{proposition}[theorem]{Proposition}
\newtheorem{problem}[theorem]{Problem}

\newtheorem{mainthm}[theorem]{Main Theorem}

\theoremstyle{definition}
\newtheorem{definition}[theorem]{Definition}
\newtheorem{remark}[theorem]{Remark}
\newtheorem{example}[theorem]{Example}

\numberwithin{equation}{section}

\DeclareMathOperator{\fix}{Fix}
\DeclareMathOperator{\per}{Per}
\DeclareMathOperator{\rec}{Rec}
\DeclareMathOperator{\fs}{FS}
\DeclareMathOperator{\cl}{cl}
\renewcommand{\int}{\mathop{\mathrm{int}}}

\begin{document}

%%%%% To ease editing, for IMPAN journals add:

\baselineskip=17pt

%%%%%%%%%%%%%%%%

\title[Equicontinuous mappings on trees]{Equicontinuous mappings on finite trees}

\author[G. Acosta]{Gerardo Acosta}
\address{Instituto de Matem\'aticas\\
Universidad Nacional Aut\'onoma de M\'exico\\
\'Area de la Investigaci\'on Cient\'{\i}fica, Circuito Exterior, Ciudad Universitaria\\ Coyoac\'an, 04510, CDMX, Mexico.
}
\email{gacosta@matem.unam.mx}

\author[D. Fern\'andez]{David Fern\'andez-Bret\'on}
\address{Instituto de Matem\'aticas\\
Universidad Nacional Aut\'onoma de M\'exico\\
\'Area de la Investigaci\'on Cient\'{\i}fica, Circuito Exterior, Ciudad Universitaria\\
Coyoac\'an, 04510, CDMX, Mexico.
}
\email{djfernandez@im.unam.mx}
\urladdr{https://homepage.univie.ac.at/david.fernandez-breton/}

\date{}

\begin{abstract}
If $X$ is a finite tree and $f \colon X \longrightarrow X$ is a map, as the Main Theorem of this
paper (Theorem~\ref{mainthm}), we find eight conditions, each of which is equivalent to the fact that $f$ is equicontinuous. To name just a few of the results obtained: the equicontinuity of $f$ is equivalent to the fact that 
there is no arc $A \subseteq X$ satisfying $A  \subsetneq f^n[A]$ for some $n\in \mathbb{N}$. It is also equivalent 
to the fact that for some nonprincial ultrafilter $u$, the function $f^u \colon X \longrightarrow X$ is continuous (in 
other words, failure of equicontinuity of $f$ is equivalent to the failure of continuity of {\it every} element of the Ellis remainder $g\in E(X,f)^*$). One of the tools used in the proofs is the Ramsey-theoretic result known as Hindman's theorem. Our results generalize the ones shown by Vidal-Escobar and Garc\'ia-Ferreira 
in~\cite{ivon-salvador}, and complement those of Bruckner and Ceder~(\cite{bruckner-ceder}), Mai~(\cite{mai}) and Camargo, Rinc\'on and Uzc\'ategui~(\cite{camargo-rincon-uzcategui}).
\end{abstract}



\maketitle

\section{Introduction}

For a metric space $X,$ this paper deals with maps $f \colon X \longrightarrow X,$ whose family of iterates is equicontinuous. Such functions represent well-behaved, non-chaotic, dynamical systems (equicontinuity 
is diametrically opposite to what is known as sensitivity to initial conditions, see \cite[Theorem~2.4]{Akin}). 

\begin{definition}\label{def:dynamicalsystem}\hfill
\begin{enumerate}
\item A {\bf discrete dynamical system} is a pair $(X,f)$ such that $X$ is a metric space and 
         $f \colon X\longrightarrow X$ is a {\bf map}, i.e. a continuous function.
\item If $(X,f)$ is a discrete dynamical system, we define $f^0$ as the identity map on $X$, and, for each 
        $n \in \mathbb N$, $f^n = f^{n-1} \circ f$.
\item If $X,Y$ are metric spaces and $\mathcal F$ is a family of functions from $X$ to $Y$, we say that 
         $\mathcal F$ is {\bf equicontinuous at} $x \in X$ if for every $\varepsilon>0$ there exists a $\delta>0$ such 
         that $d(x,y) < \delta$ implies $d(f(x),f(y))\leq\varepsilon$ for all $y\in X$ and every $f\in\mathcal F;$ and if 
         $\mathcal F$ is equicontinuous at every $x\in X$, we say that
         $\mathcal F$ is {\bf equicontinuous}.
\item If $X$ is a metric space, the function $f \colon X\longrightarrow X$ is  {\bf equicontinuous at $x\in X$} if its 
         family of iterates, $\{f^n\big|n\in\mathbb N\}$, is equicontinuous at $x$; and if $f$ is equicontinuous at every 
         $x\in X$ we say that it is {\bf equicontinuous}.
\end{enumerate}
\end{definition}

The definition of equicontinuity makes sense for every uniform space, but in this paper we will only consider metric spaces. Note that, upon fixing $x,\varepsilon,\delta$, equicontinuity of a family of functions $\mathcal F$ is a pointwise closed condition; consequently if $\mathcal F$ is equicontinuous at $x$ then so is $\overline{\mathcal F}$, where $\overline{\mathcal F}$ is the closure of $\mathcal F$ in $Y^X$ with the product topology. Note also that, if $X$ is compact, then by the usual argument, equicontinuity implies uniform equicontinuity (i.e., given 
$\varepsilon>0$, a $\delta>0$ can be chosen to work for all $x\in X$).

\begin{definition}\label{def:ellis}
Let $(X,f)$ be a discrete dynamical system, where $X$ is compact.
\begin{enumerate}
\item The {\bf Ellis semigroup} (also called the {\bf enveloping semigroup}) of $(X,f)$ is defined as 
         $E(X,f)=\overline{\{f^n\big|n\in\mathbb N\}}$, the closure in $X^X$ (with the product topology) of the family 
         $\{f^n\big|n\in\mathbb N\}$. Note that, as $X^X$ is compact (by Tychonoff's theorem), so is $E(X,f)$.
\item The {\bf Ellis remainder} of the discrete dynamical system $(X,f)$ is 
         $$
         E(X,f)^*=  \bigcap_{n=1}^\infty\overline{\{f^k \big| k\geq n\}}.
         $$ 
         Note that $E(X,f) = E(X,f)^* \cup \{f^n \big|n \in \mathbb N\}$.
\end{enumerate}
\end{definition}

Composition of functions is what makes $E(X,f)$ a {\em semigroup}. In fact, $E(X,f)$ is a compact right-topological semigroup (a semigroup equipped with a topology making all right translations 
continuous). Since $X$ is a metric space, by the observation immediately after Definition~\ref{def:dynamicalsystem}, equicontinuity of $f$ is equivalent to equicontinuity of the family $E(X,f)$, and either of these is equivalent to the same statement with uniform equicontinuity instead of equicontinuity (cf.~\cite[Theorem~3.3]{garcia-sanchis}).

The seemingly abstract object $E(X,f)$ can be made more concrete by means of ultrafilters: for every ultrafilter $u$ on $\mathbb N$, define the ultrafilter-limit function $f^u$ by letting $f^u(x)=u\text{-}\lim_{n\in\mathbb N}f^n(x)$. Then by \cite[Theorem~2.2]{garcia-sanchis} we have 
$$
E(X,f)=\{f^u\big|u\text{ is an ultrafilter on }\mathbb N\},
$$
and consequently
$$
E(X,f)^*=\{f^u\big|u\text{ is a nonprincipal ultrafilter on }\mathbb N\}.
$$
Full definitions of ultrafilters, both principal and nonprincipal, as well as of $u$-limits will be provided in Section~\ref{sect:ultrafilters}. 

\begin{definition}\label{def:fix}
Let $(X,f)$ be a discrete dynamical system.
\begin{enumerate}
\item $x\in X$ is a {\bf fixed point} if $f(x)=x$; the set of fixed points of $f$ is denoted by $\fix(f).$
\item $x\in X$ is a {\bf periodic point} if $f^n(x)=x$, for some $n\in\mathbb N$, in which case the least such 
        $n$ is called the {\bf period} of $x$; the set of all periodic points of $f$ is denoted by $\per(f)$.
\item $f$ is {\bf periodic} if there exists $n \in \mathbb{N}$ such that $f^n$ is the identity map on $X$,  
         and $f$ is {\bf pointwise-periodic} if $\per(f) = X.$
\end{enumerate}
\end{definition}

It is immediate from Definition~\ref{def:fix} that $\per(f)=\bigcup_{n=1}^\infty\fix(f^n)$.

We primarily deal with {\em continua} (compact, connected and metric spaces). A {\em simple closed curve} is a continuum homeomorphic to the unit circle $\mathbb S^1,$ and an {\em arc} is a continuum homeomorphic to the
unit interval $[0,1].$ Other examples of continua are {\em finite graphs} (compact, connected one-dimensional polyhedra),  {\em dendrites}, {\em finite trees} and {\em $k$-ods}, for each $k \in \mathbb N$ with $k \geq 2.$ We give the proper definitions of the last three in Section~2. For the moment it is convenient to note that a $2$-od is an arc,
$k$-ods are finite trees and finite trees are dendrites.

In the early nineties, Bruckner and Hu (\cite{bruckner-hu}) and Bruckner and Ceder (\cite{bruckner-ceder}) carried 
out a very deep and complete study of equicontinuity of maps defined on arcs, obtaining the following result.

\begin{theorem}[Subset of~{\cite[Theorem 1.2]{bruckner-ceder}}]\label{thm:bruck-ced}
If $X$ is an arc and $f \colon X\longrightarrow X$ is a map, then the following are equivalent:
\begin{enumerate}
\item[\emph{(1)}] $f$ is equicontinuous;
\item[\emph{(2)}]  the restriction $f^2\upharpoonright\bigcap_{m=1}^\infty f^m[X]$ is the identity map;
\item[\emph{(3)}]  $\fix(f^2)=\bigcap_{m=1}^\infty f^m[X];$
\item[\emph{(4)}]  $\fix(f^2)$ is connected.
\end{enumerate}
\end{theorem}

Attempting to generalize this result from arcs to finite trees is futile, if taken too literally. Allowing, however, exponents other than 2 in the theorem above yields valid characterizations: we prove that, for an arbitrary finite tree $X$ and a map $f \colon X\longrightarrow X$, equicontinuity of $f$ is equivalent to each of the following conditions: that the restriction $f\upharpoonright\bigcap_{m=1}^\infty f^m[X]$ is periodic, that 
$\fix(f^n)=\bigcap_{m=1}^\infty f^m[X]$ for some $n$, and that $\fix(f^n)$ is connected for all $n$; furthermore, any of these is also equivalent to the set $\per(f)$ being connected.

Further interesting results regarding equicontinuity of a map $f \colon X\longrightarrow X$ have been obtained by Mai (\cite{mai}) in the case where $X$ is a finite graph, and by Camargo, Rinc\'on and 
Uzc\'ategui (\cite{camargo-rincon-uzcategui}) in the case where $X$ is a dendrite. The former shows in 
\cite[Theorem 5.2]{mai} that, if $X$ is a finite graph, then $f$ is equicontinuous if and only if 
$\bigcap_{m=1}^\infty f^m[X] = \rec(f)$ (here $\rec(f)$ denotes the set of recurrent points of $f$, to be defined later in Definition~\ref{def:moved}; for the moment just note that $\per(f)\subseteq\rec(f)$); the latter  proves in 
\cite[Theorem~4.12]{camargo-rincon-uzcategui} that, if $X$ is a dendrite, then $f$ is equicontinuous if and only if 
$\cl_X(\per(f))=\bigcap_{m=1}^\infty f^m[X]$ plus an extra condition having to do with the $\omega$-limit sets of $f$. Obtaining a simultaneous strengthening of these two results at the expense of considering a less general class of spaces, we prove that, in the case where $X$ is a finite tree, $f$ is equicontinuous if and only if $\per(f)=\bigcap_{m=1}^\infty f^m[X]$.

Another concept that will play a central role in this paper is that of an expanding arc. To motivate this concept consider a nonnegative $\alpha\in\mathbb R$ and the map
$f_\alpha \colon \mathbb R\longrightarrow\mathbb R$ defined by
\begin{equation}\label{exe1}
f_\alpha(x)=\alpha x, \hspace{.5cm} \mbox{for each } x \in X.
\end{equation}
It is readily checked that $f_\alpha$ is equicontinuous if and only if $0\leq\alpha\leq 1$, whereas if $\alpha>1$ then $f_\alpha$ fails to be equicontinuous at every $x\in\mathbb R$. Intuitively speaking, maps that expand the real line fail to be equicontinuous. Note that for the map $f_\alpha$ defined in (\ref{exe1}), with $\alpha>1$, we have 
$I\subsetneq f_\alpha^n[I]$ for all $n\in\mathbb N$, where $I=[0,1]$ is the unit interval. This leads to the following definition.

\begin{definition}
Let $(X,f)$ be a discrete dynamical system, and let $A \subseteq X$ be an arc. We say that $A$ is an 
{\bf $f$-expanding arc} if there exists $n\in\mathbb N$ such that $A \subsetneq f^n[A]$.
\end{definition}

Hence the map $f_\alpha$ defined in (\ref{exe1}) fails to be equicontinuous if and only if $[0,1]$ is 
$f_\alpha$-expanding. Surprisingly, something like this very simple characterization still holds in more general situations. Namely, in~\cite[Theorems~3.1 and 3.7]{ivon-salvador}, Vidal-Escobar and Garc\'{\i}a-Ferreira established the following result (in the case $k=2$ they assume $f$ is surjective, but the general case follows from the proof of~\cite[Theorem~1.2]{bruckner-ceder}):

\begin{proposition}\label{simple-f-expanding}
Let $X$ be a $k$-od for some $k\geq 2$ and $f:X\longrightarrow X$ be a map where $f[X]$ is not a singleton. Then $f$ is not equicontinuous if and only if $X$ contains an $f$-expanding arc.
\end{proposition}

In this paper, we generalize Proposition~\ref{simple-f-expanding} from $k$-ods to arbitrary finite graphs. Our proof 
of this generalization uses at a crucial point a highly nontrivial Ramsey-theoretic result (Hindman's theorem). 
Hence, our result is not a direct use of the proofs presented in \cite[Theorem 1.2]{bruckner-ceder} and
\cite[Theorems~3.1 and 3.7]{ivon-salvador}.
 
Another result of Vidal-Escobar and Garc\'{\i}a-Ferreira~(\cite[Theorem~3.7]{ivon-salvador}) is that 
for each map $f:X\longrightarrow X$, where $X$ is a $k$-od with $k\geq 3$, if $f^u$ is continuous for {\em every} nonprincipal ultrafilter $u$ then $f$ is equicontinous (note that the converse implication is trivially true as a consequence of the observation right after Definition~\ref{def:ellis}). Hence if $f$ is not equicontinuous then for some nonprincipal ultrafilter $u,$ $f^u$ is not continuous. In~\cite{ivon-salvador} the authors consider the possibility that, for some nonprincipal ultrafilter $v,$ distinct from $u,$ $f^v$ might be continuous.

In this paper we show that the possibility mentioned in the previous line cannot occur by proving that the statement from the preceding paragraph is true with {\em every} replaced by {\sl some}, even if $X$ is a finite tree rather than just a $k$-od.  As a consequence of this, if $f$ fails to be equicontinuous with $X$ a finite tree, then 
{\em every} element $g\in E(X,f)^*$ fails to be continuous. Thus, for maps 
$f \colon X\longrightarrow X$ on a finite tree $X$, we have a strong dichotomy by means of which either every element $g\in E(X,f)^*$ is continuous, or every element $g\in E(X,f)^*$ is discontinuous, according to whether or not $f$ is equicontinuous. This is a direct generalization of a result of Szuca~(\cite[Theorem~2]{szuca}), who obtains the same dichotomy for maps in an arc. This result is therefore worth stating explicitly.

\begin{theorem}
Let $(X,f)$ be a discrete dynamical system, where $X$ is a finite tree. Then, either every element of $E(X,f)^*$ is continuous, or every element of $E(X,f)^*$ is discontinuous.
\end{theorem}

We now state the Main Theorem of this paper.

\begin{mainthm}\label{mainthm}
Let $X$ be a finite tree and $f:X\longrightarrow X$ be a map. Then, the following are equivalent:
\begin{enumerate}
\item[\emph{(a)}] $f$ is equicontinuous;
\item[\emph{(b)}] there is an $n\in\mathbb N$ such that the restriction of $f^n$ to $\bigcap_{m=1}^\infty f^m[X]$ 
        is the identity map;
\item[\emph{(c)}] there exists an $n\in\mathbb N$ such that $\fix(f^n)=\bigcap_{m=1}^\infty f^m[X];$
\item[\emph{(d)}] $\per(f)=\bigcap_{m=1}^\infty f^m[X];$
\item[\emph{(e)}] there is no $f$-expanding arc in $X;$
\item[\emph{(f)}] for every $n\in\mathbb N$, the set $\fix(f^n)$ is connected;
\item[\emph{(g)}] the set $\per(f)$ is connected;
\item[\emph{(h)}] for every nonprincipal ultrafilter $u$, the function $f^u$ is continuous \emph{(}i.e., every element of 
        $E(X,f)^*$ is continuous\emph{)};
\item[\emph{(i)}] for some nonprincipal ultrafilter $u$, the function $f^u$ is continuous \emph{(}i.e., some element of 
        $E(X,f)^*$ is continuous\emph{)}.
\end{enumerate}
\end{mainthm}

\begin{remark}\label{aftermainthm}
Some remarks about the equivalences from Theorem~\ref{mainthm}:

\begin{enumerate}
\item The equivalence between (e), (f) and (g) will be established not only for finite trees, but for arbitrary dendrites.
\item The equivalence between (a), (e) and (h) shows that \cite[Theorem~3.7]{ivon-salvador} can be extended
         from $k$-ods to finite trees. This is a partial answer to \cite[Question~3.10]{ivon-salvador}.
\item The equivalences $(\mathrm{a})\iff(\mathrm{x})$, where $\mathrm{x}\in\{\mathrm{b},\mathrm{c},\ldots,\mathrm{i}\}$, fail if we allow $X$ to be an arbitrary dendrite, and 
         even if $X$ is merely a dendrite all of whose branching points are of finite order (see \S 4.2), or (with the 
         possible exception of $(\mathrm{a})\iff(\mathrm{d})$, see Remark~\ref{rem:aiffd}) if $X$ is merely a dendrite with finitely many 
         branching points (see \S 4.1).
\item Further conditions equivalent to equicontinuity of a map $f$ on a space $X$ have been established 
         in~\cite[Theorem~2, p.~62]{sun2} for $X$ a finite tree, in~\cite[Theorem~5.2]{mai} when $X$ is a finite graph, 
         and in~\cite[Theorem~4.12]{camargo-rincon-uzcategui} in the case of $X$ an arbitrary dendrite.
\end{enumerate}
\end{remark}

The paper is structured around the equivalence that constitutes its main result (Theorem~\ref{mainthm}). In Section~2 we begin by proving the equivalence of items (a), (b), (c) and (d), which is a fairly elementary result, and the rest of the section is devoted to the study of expanding arcs, starting with the equivalence of (e) and (f), and concluding with the implication from (e) to (a). Then, in Section~3, we establish the equivalence between (e) and (g), in order to later on focus on ultrafilter-limit functions to establish that (i) implies (e) (this finishes the main theorem, since the implication from (h) to (i) is obvious and that from (a) to (h) is well-known). Finally, in Section~4 we 
describe the examples that exhibit the failure of all these characterizations in the context of arbitrary dendrites, and state some questions that remain open.

%\section*{Acknowledgements}
%We are grateful with both the anonymous referee and an editor of the journal, for a careful reading and a nontrivial amount of suggestions that helped us to improve this paper.

\section{Equicontinuity and expanding arcs}

Given a subset $A$ of a space $X,$ we denote by either $\overline{A}$ or $\cl_X(A),$ the closure of $A$ in $X.$ 
The interior of $A$ in $X$ is denoted by $\int_X(A).$ We begin by stating some standard results that will be used. 

\begin{proposition}\label{lem:referee}
Let $(A_n)_{n\in\mathbb N}$ be a decreasing sequence of closed subsets of a compact space $X$, and let $A=\bigcap_{n=1}^\infty A_n$. Then,
\begin{enumerate}
\item[\emph{(1)}] if $U$ is an open set containing $A$, then $A_n\subseteq U$ for all sufficiently large $n$;
\item[\emph{(2)}] if each $A_n$ is nonempty, then so is $A$;
\item[\emph{(3)}] if every $A_n$ is connected, then so is $A$;
\item[\emph{(4)}] if $f:X\longrightarrow Y$ is a map, then $\bigcap_{n=1}^\infty f[A_n]=f[A]$.
\end{enumerate}
\end{proposition}
\begin{proof}
Parts (1) and (3) follow from \cite[Corollary~3.1.5 and Corollary~6.1.19]{engelking}. Part (2) follows from
(1) with $U = \varnothing.$ To show part (4), it is enough to verify that $\bigcap_{n=1}^\infty f[A_n] \subseteq f[A]$.
Let $y \in \bigcap_{n=1}^\infty f[A_n].$ Since $(f^{-1}[y] \cap A_n)_{n\in\mathbb N}$ is a decreasing sequence of closed, nonempty subsets of $X$; by (2), $f^{-1}[y] \cap A \ne \varnothing$ and then $y \in f[A].$
\end{proof}

If $(X,f)$ is a discrete dynamical system with $X$ compact, then by Proposition~\ref{lem:referee}, 
$\bigcap_{m=1}^\infty f^m[X]$ is a nonempty compact subspace of $X$ satisfying 
$$
f\left[\bigcap_{m=1}^\infty f^m[X]\right]=\bigcap_{m=1}^\infty f^m[X].
$$ 
This means that the restricted map $f\upharpoonright\bigcap_{m=1}^\infty f^m[X]$ is onto $\bigcap_{m=1}^\infty f^m[X]$. In the case where $X$ is a connected space, so is $\bigcap_{m=1}^\infty f^m[X]$, again by Proposition~\ref{lem:referee}.

\subsection{Basic lemmas, definitions, and the first equivalences}

Before delving deep into the study of dendrites and finite trees, we state two general lemmas (on arbitrary metric spaces) containing some useful consequences of the failure of equicontinuity of a map. First note that a map $f \colon X\longrightarrow X$ fails to be equicontinuous at the point $x\in X$ if and only if there exists an $\varepsilon>0$, a sequence of points $(x_k)_{k\in\mathbb N}$ converging to $x$, and an increasing sequence of indices $(n_k)_{k\in\mathbb N}$ such that $d(f^{n_k}(x_k),f^{n_k}(x)) > \varepsilon$ for all 
$k\in\mathbb N$. In this case we will say that $\varepsilon$, $(x_k)_{k\in\mathbb N}$, and $(n_k)_{k\in\mathbb N}$ {\em witness} the failure of equicontinuity of $f$ at $x$.

\begin{lemma}\label{lem:f-bad-fn-bad}
Let $X$ be a metric space, let $f \colon X\longrightarrow X$ be a map, and suppose that $f$ fails to be equicontinuous at $x\in X$. Then, for every $n\in\mathbb N$, 
\begin{enumerate}
\item[\emph{(1)}] $f$ fails to be equicontinuous at $f^n(x)$, and 
\item[\emph{(2)}] there exists an $0 \leq i<n$ such that $f^n$ fails to be equicontinuous at $f^i(x)$.
\end{enumerate}
\end{lemma}
\begin{proof}
Suppose that $\varepsilon>0$, the sequence of points $(x_k)_{k\in\mathbb N}$, and the sequence of indices 
$(n_k)_{k\in\mathbb N}$ witness the failure of equicontinuity of $f$ at $x$, and let $n\in\mathbb N$. To prove (1), assume without loss of generality that $n_1>n$; then, the sequence $(f^n(x_k))_{k\in\mathbb N}$ (which converges to $f^n(x)$ by continuity of the function $f^n$), and the increasing sequence $(n_k-n)_{k\in\mathbb N}$ of natural numbers, along with $\varepsilon$, witness the failure of equicontinuity of $f$ at $f^n(x)$. This shows (1). For (2), apply the pigeonhole principle to assume, without loss of generality, that there is a fixed $0 \leq i<n$ such that $n_k\equiv i\mod n$ for all $k\in\mathbb N$. Let $m_k$ be such that $n_k=n m_k+i$; then, the sequence $(f^i(x_k))_{k\in\mathbb N}$, which converges to $f^i(x)$, along with the increasing sequence 
$(m_k)_{k\in\mathbb N}$ of natural numbers and $\varepsilon$, witness the failure of equicontinuity of $f^n$ at 
$f^i(x)$.
\end{proof}

Before considering the specific case of dendrites, we introduce some more definitions and a general result that 
shall be used later.

\begin{definition}\label{def:moved}
Let $X$ be a metric space, and let $f \colon X\longrightarrow X$ be a map.
\begin{enumerate}
\item The {\bf $\omega$-limit set} of $f$ at $x \in X,$ is the set of all points $y \in X$ for which there is an increasing 
         sequence $(n_i)_{i \in \mathbb N}$ with $\lim_{i \to \infty}f^{n_i}(x) = y$; this set is denoted by $\omega(x,f)$.
\item A point $x\in X$ is a {\bf recurrent point} if $x\in\omega(x,f)$; the set of recurrent points of $f$ is denoted by
         $\rec(f)$.
%\item The function $f$ is called {\bf pointwise-recurrent}  if $\rec(f) = X$.
\end{enumerate}
\end{definition}

It is immediate that every periodic point is recurrent; the converse is not necessarily true. The next proposition follows from known results and will be used for our Main Theorem.

\begin{proposition}\label{nuevo1}
If $X$ is a finite tree and $f \colon X\longrightarrow X$ is an equicontinuous surjective map, then $f$ is a homeomorphism which furthermore is periodic.
\end{proposition}
\begin{proof}
By~\cite[Proposition~2.4]{mai} and~\cite[Corollary~8]{bruckner-hu}, cf.~\cite[Corollary~3.2]{mai}, $f$ is a homeomorphism that is pointwise-recurrent, i.e., such that $x \in \omega (x,f)$ for each $x \in X$. 
Hence, by~\cite[Theorem~4.4]{mai2}, $f$ is periodic.
\end{proof}

We now mention some standard facts about dendrites that will be used throughout the paper.

\begin{definition}
A {\bf dendrite} is a locally connected continuum without simple closed curves. 
\end{definition}

A map of a dendrite into itself has a fixed point (\cite[Theorem~10.31]{nadler}). Every subcontinuum of a dendrite is again a dendrite (\cite[Corollary~10.6]{nadler}), and every connected subset of a dendrite is arcwise connected (\cite[Proposition~10.9]{nadler}). If $X$ is a dendrite, $x,y\in X$ and $x \ne y,$ then there is a unique (closed) arc in $X$ joining $x$ and $y$; such an arc will always be denoted by $xy$. Since continuous images of connected sets must be connected, for any map $f \colon X\longrightarrow X$ and every $x,y\in X$ we have that $f(x)f(y)\subseteq f[xy]$.

Whenever $X$ is a dendrite and $Y$ is a subcontinuum of $X$, then there exists a retraction 
$r_Y \colon X\longrightarrow Y,$ called the {\em first point function}, such that for $x\in X$ and $y\in Y$, $r_Y(x)$ is the first point in the arc $xy$ (equipping such an arc with a linear order where $x\leq y$) that belongs to $Y.$ The mapping $r_Y$ does not depend on the specific $y\in Y$ (see~\cite[Lemmas~10.24, 10.25 and Terminology 10.26]{nadler}).

Finally, every dendrite is {\em uniformly locally arcwise connected}, that is, for every $\varepsilon>0$ there exists a $\delta>0$ such that, whenever $d(x,y)<\delta$ and $x \ne y$, the arc $xy$ must have diameter $<\varepsilon$ (as a matter of fact, every compact, connected and locally connected metric space has this property, which in this more general context must be phrased as: for every $\varepsilon>0$ there exists a $\delta>0$ such that whenever $d(x,y)<\delta$ and
$x \ne y,$ then there is an arc joining $x$ and $y$ with diameter $<\varepsilon$; see~\cite[Theorem~31.4]{willard}).

\begin{proposition}\label{fixpoint}
Let $X$ be a dendrite, let $f \colon X\longrightarrow X$ be a map, and let $x\in X$. If $Y\subseteq X \setminus \{x\}$ is a connected component of $X\setminus \{x\}$ such that $f(x)\in Y$, then 
$Y\cap\fix(f)\neq\varnothing$.
\end{proposition}
\begin{proof}
Notice that $\cl_X(Y)=Y\cup\{x\}$ is a subcontinuum of $X$ --hence $\cl_X(Y)$ is itself a dendrite. We consider the first point function $r_{\cl_X(Y)}:X\longrightarrow \cl_X(Y)$ and note that, for $y\notin \cl_X(Y)$, it must be the case that $r_{\cl_X(Y)}(y)=x$. Since $\cl_X(Y)$ is a dendrite, it has the fixed point property; therefore the map $r_{\cl_X(Y)}\circ (f\upharpoonright \cl_X(Y)) \colon \cl_X(Y)\longrightarrow \cl_X(Y)$ has a fixed point $y$. 
It is now easy to check that we must have $f(y)=y$.
\end{proof}

The following is another definition that will be crucial throughout the paper.

\begin{definition}\label{order}
Let $X$ be a dendrite and $k \in \mathbb N.$
\begin{enumerate}
\item The {\bf order} of a point $x\in X$ is the number of connected components of $X\setminus\{x\}$;
\item a point $x\in X$ is an {\bf endpoint} if its order is $1$, and a {\bf branching point} if its order is $\geq 3$;
\item $xy \subseteq X$ is a {\bf free arc} in $X,$ if no element of $xy \setminus\{x,y\}$ is a branching point;
\item for $k\geq 3$, $X$ is a {\bf $k$-od} if it contains exactly one branching point (called the {\bf vertex} of $X$), 
         which has order $k$; a {\bf $2$-od} is simply defined to be an arc (we do not specify a vertex in this case);
\item $X$ is a {\bf finite tree} if it has only finitely many branching points and each of these branching points has a 
         finite order.
\end{enumerate}
\end{definition}

Note that $xy \subseteq X$ is a free arc in $X$ if and only if $xy \setminus\{x,y\}$ is open in $X.$ In a general topological space $X$, the order of a point $x\in X$ is defined as the least cardinal number $\kappa$ such that, for every open neighbourhood $U$ of $x,$ there exists another open neighbourhood $V$ with $x\in V\subseteq U$ and $|\partial(V)|\leq\kappa$ (where $\partial(V)$ denotes the boundary of $V$ in $X$), cf.~\cite[Definition~9.3]{nadler}; this will be important towards the end of Section~4. If, however, the topological space $X$ under consideration is a dendrite, then Definition~\ref{order} agrees with the general definition just mentioned 
(see~\cite[Lemma~10.12, Theorem~10.13 and Corollary~10.20.1]{nadler}).

We now show the equivalence of the first four conditions in Theorem~\ref{mainthm}.

\begin{proposition}\label{cor:firstfour}
Let $X$ be a finite tree and let $f \colon X\longrightarrow X$ be a map. Then the following conditions are equivalent:
\begin{enumerate}
\item[\emph{(a)}] $f$ is equicontinuous;
\item[\emph{(b)}] for some $n\in\mathbb N$, the restriction $f^n\upharpoonright\bigcap_{m=1}^\infty f^m[X]$ is 
        the identity map;
\item[\emph{(c)}] for some $n\in\mathbb N$, $\fix(f^n)=\bigcap_{m=1}^\infty f^m[X]$;
\item[\emph{(d)}] $\per(f)=\bigcap_{m=1}^\infty f^m[X]$.
\end{enumerate}
\end{proposition}
\begin{proof}
We consider first the case where $f$ is surjective. Note that in such situation, $X=\bigcap_{m=1}^\infty f^m[X]$
and $f^n\upharpoonright\bigcap_{m=1}^\infty f^m[X] = f^n$ for each $n\in\mathbb N$. Moreover (b) asserts
that $f$ is periodic, (c) that $\fix(f^n)= X$ for some $n\in\mathbb N$, and (d) that $f$ is pointwise-periodic.
By Proposition~\ref{nuevo1}, $(a) \Rightarrow (b)$.  The implications $(b) \Rightarrow (c) \Rightarrow (d)$ are
obvious and, by ~\cite[Theorem~4.14]{camargo-rincon-uzcategui}, the implication $(d) \Rightarrow (a)$
holds not only on finite trees, but on every dendrite and with $f$ being any (surjective) map.

We now consider the case of an arbitrary (not necessarily surjective) map 
$f \colon X\longrightarrow X.$ Since every finite tree is, in particular, a finite graph, we may 
use~\cite[Theorem 5.2]{mai} to see that $f$ is equicontinuous if and only if so is 
$f\upharpoonright\bigcap_{m=1}^\infty f^m[X]$, and since the latter map is onto 
$\bigcap_{m=1}^\infty f^m[X]$ (and since $\fix(f^n)=\fix(f^n\upharpoonright\bigcap_{m=1}^\infty f^m[X])$ and 
$\per(f)=\per(f\upharpoonright\bigcap_{m=1}^\infty f^m[X])$), then the theorem follows from the surjective case.
\end{proof}

\subsection{Expanding arcs}

We now analyze some implications of the existence of expanding arcs.

\begin{lemma}\label{arc-iff-disconnection}
Let $X$ be a dendrite and let $f \colon X\longrightarrow X$ be a map. Then the following are equivalent:
\begin{enumerate}
\item[\emph{($e^\prime$)}] $X$ contains an $f$-expanding arc;
\item[\emph{($f^\prime$)}] for some $n\in\mathbb N$, the set $\fix(f^n)$ is disconnected;
\item[\emph{($j^\prime$)}] there exist points $x,y\in X$ and $n\in\mathbb N$ such that $x=f^n(x)$, 
       $y\neq f^n(y)$, and $y\in x f^n(y)$.
\end{enumerate}
\end{lemma}
\begin{proof}\hfill

\noindent $(e^\prime) \Rightarrow (f^\prime)$ Let $ab$ be an $f$-expanding arc and fix an $n\in\mathbb N$ such that $ab\subsetneq f^n[ab]$. We consider three cases according to whether both, exactly one, or none of $a,b$ are fixed points for $f^n$. \vskip .2cm
\begin{description}
\item[Case 1] If $f^n(a)=a$ and $f^n(b)=b$, use the fact that $ab\subsetneq f^n[ab]$ to get a $c\in ab\setminus\{a,b\}$ such that $f^n(c)\notin ab$. In particular, $c\notin\fix(f^n)$ with $c\in ab$ and $a,b\in\fix(f^n)$, showing that $\fix(f^n)$ is not arcwise connected, so $\fix(f^n)$ is disconnected.\vskip .2cm

\item[Case 2] If $f^n(a)=a$ but $f^n(b)\neq b$, use the fact that $ab\subseteq f^n[ab]$ to get a $c\in ab\setminus\{a,b\}$ such that $f^n(c)=b$. Let $Z$ be the connected component of $X\setminus\{c\}$ containing 
$a.$ If $b \in Z$ then, since $Z$ is arcwise connected, we have $c \in ab \subseteq Z \subseteq X\setminus\{c\},$ a contradiction. Hence, $b \notin Z$. Letting $Y \ne Z$ be the connected component of $X\setminus\{c\}$ containing $b$, Proposition~\ref{fixpoint} guarantees the existence of a $d\in\fix(f^n)\cap Y$. Then we have $c\notin\fix(f^n)$, 
$a,d \in \fix(f^n)$, and $c\in ad$, showing that $\fix(f^n)$ is disconnected. \vskip .2cm

\item[Case 3]  If $f^n(a)\neq a$ and $f^n(b)\neq b$. Then (since $ab\subseteq f^n[ab]$) we can find 
$x,y\in ab\setminus\{a,b\}$ with $f^n(x)=a$ and $f^n(y)=b$. Note that $x \ne y$. Equip $ab$ with a 
linear order via a homeomorphism $:[0,1]\longrightarrow ab$ mapping $0$ to $a$ and $1$ to $b$. If $x< y$, then a couple of applications of Proposition~\ref{fixpoint} yield fixed points $c,d\in\fix(f^n)$ such that $x,y\in cd$; since $x,y\notin\fix(f^n)$, this shows that $\fix(f^n)$ is disconnected. If, on the other hand, we have $y < x$, then notice 
that $y\in ab=f^n(x)f^n(y)\subseteq f^n[xy]$, so there is a $y^\prime \in xy$ with $f^n(y^\prime)=y$; also, 
$x\in yb=f^n(y')f^n(y)\subseteq f^n[yy']$ and so there is an $x^\prime \in yy^\prime$ with $f^n(x')=x$. This way we have obtained $x^\prime < y^\prime$ with $f^{2n}(x^\prime)=f^n(x)=a$ and $f^{2n}(y^\prime)=f^n(y)=b$, thus, a couple of applications of Proposition~\ref{fixpoint} yield two fixed points $c,d\in\fix(f^{2n})$ with 
$x^\prime,y^\prime \in cd$; the fact that $x^\prime,y^\prime \notin\fix(f^{2n})$ implies then that $\fix(f^{2n})$ is disconnected, and we are done.
\end{description}

\noindent $(f^\prime) \Rightarrow (j^\prime)$ Let $n\in\mathbb N$ and $a,b\in\fix(f^n)$ be such that $ab\nsubseteq\fix(f^n)$. Then there is a $y\in ab$ with $f^n(y)\neq y$. Considering the first point function $r_{ab} \colon X\longrightarrow ab$ and the point $z=r_{ab}(f^n(y))$, we must have that either $y\in az$ or $y\in bz$. Since the situation is entirely symmetric, assume without loss of generality that $y\in az$ and let $x=a$. Then we have $f^n(x)=x$ and $y\in xz\subseteq xz\cup zf^n(y)=xf^n(y)$.

\noindent $(j')\Rightarrow (e^\prime)$ Under the assumptions we have $xy\subsetneq x f^n(y)=f^n(x) f^n(y)\subseteq f^n[xy]$ and so $xy$ is an $f$-expanding arc.
\end{proof}

We now proceed to prove that the failure of equicontinuity of a map on a finite tree implies the existence of an 
$f$-expanding arc. The next result allows us to restrict any map without expanding arcs from a 
finite tree to a simpler subcontinuum.

\begin{lemma}\label{tree-to-dendrite}
Let $X$ be a finite tree, let $f \colon X\longrightarrow X$ be a map and let $x \in \fix(f).$ If $X$ has 
no $f$-expanding arcs, then there is an $f$-invariant subspace $Y\subseteq X$ \emph{(}that is, $f[Y]\subseteq Y$\emph{)}, where $Y$ is a $k$-od for some $k \geq 2$, such that $x \in \int_X(Y)$.
\end{lemma}

\begin{proof}
Let $n$ be the order of $x\in X.$ Since $X$ is a finite tree, for some $\varepsilon>0$ the closed ball centred at $x$ 
is a union of finitely many arcs, say $I_1,\ldots, I_n$, which pairwise intersect at $x$ only. By continuity 
of $f,f^2,\ldots,f^n$, we can pick a $\delta$ with $0<\delta\leq\varepsilon$ such that, if $d(y,z)\leq\delta$, then $d(f^i(y),f^i(z))<\varepsilon$ for every $0 \leq i\leq n$. Let $Z$ be the closed ball of radius $\delta$ centred at $x$, and define 
$$
Y=Z\cup f[Z]\cup\cdots\cup f^n[Z].
$$ 
Due to our choice of $\delta$ and $Z$, we have $x\in\int_X(Y)\subseteq Y\subseteq\bigcup_{i=1}^n I_i$. Since $x$ is a fixed point for $f$, whenever $y\in Y$ we must have $xy\subseteq Y$. Hence $Y$ is a $k$-od for some $k \geq 2$ ($k = n$ and $x$ is the vertex of $Y$ if $n\geq 3$, $k = 2$ if $Y$ is an arc, i.e., $n \in \{1,2\};$ moreover $x$ is an interior point of the arc $Y$ if $n=2$, and it is an endpoint of the arc $Y$ if $n=1$). It remains to show that $Y$ is 
an $f$-invariant subspace, so let $y\in Y$ and let us argue that $f(y)\in Y$. We can write $y=f^m(z)$ for some $z\in Z$ (this includes the case $m=0$, interpreted as $y=z\in Z$); there is nothing to do if $m<n$, so assume that $y=f^n(z)$ for $z\in Z$. Looking at the finite sequence of points $z,f(z),\ldots,f^n(z)$, the pigeonhole principle guarantees the existence of $0\leq i<j\leq n$ and $l\in\{1,\ldots,n\}$ such that $f^i(z),f^j(z)\in I_l$. Let us linearly 
order the arc $I_l$ by copying the order of $[0,1]$ along a homeomorphism mapping $x$ to $0$. If $f^i(z)<f^j(z)$, 
this would mean that 
$$
x f^i(z)\subsetneq x f^j(z)=f^{j-i}(x)f^{j-i}(f^i(z))\subseteq f^{j-i}[x f^i(z)]
$$ 
\noindent is an $f$-expanding arc, a contradiction. Therefore we must have $f^j(z)\leq f^i(z)$, implying that 
$$
f^j(z)\in x f^i(z)=f^i(x) f^i(z)\subseteq f^i[xz]
$$ 
\noindent and so there exists a point $z^\prime \in xz\subseteq Z$ with $f^j(z)=f^i(z^\prime)$. Hence 
$y=f^n(z)=f^{n-j+i}(z^\prime)$ and so $f(y)=f^{n-(j-i)+1}(z^\prime)\in Y$, and we are done.
\end{proof}

\begin{remark}\label{non-equicontinuity-restriction}
Suppose that $X$ is a metric space, $f \colon X\longrightarrow X$ is a map, and $Y\subseteq X$ is an $f$-invariant subspace with $x \in \int_X(Y)$. If $f$ is not equicontinuous at $x$, then the restriction 
$f\upharpoonright Y \colon Y\longrightarrow Y$ also fails to be equicontinuous at $x$: given $\varepsilon>0$, if $\delta>0$ witnesses the equicontinuity of $f\upharpoonright Y$ at $x$, then $\min\{\delta,d(x,X\setminus\int_X(Y))\}$ witnesses the equicontinuity of $f \colon X\longrightarrow X$ at $x$.
\end{remark}

For our next proof we will use a Ramsey-theoretic result known as Hindman's theorem, so we proceed to explain 
the relevant concepts and notations. Given a sequence $(n_k)_{k\in\mathbb N}$ of elements of $\mathbb N$, 
its {\em set of finite sums} is defined as
\begin{eqnarray*}
\fs(n_k)_{k\in\mathbb N} & = & \left\{\sum_{k\in F}n_k\bigg|F\subseteq\mathbb N\text{ is finite and 
nonempty}\right\} \\
 & = & \left\{n_{k_1}+\cdots+n_{k_m}\big|m\in\mathbb N\text{ and }k_1<\cdots<k_m\right\},
\end{eqnarray*}
the set of all numbers that can be obtained by adding a finite amount of terms of the sequence 
$(n_k)_{k\in\mathbb N}$ without repetitions. The result known as Hindman's 
theorem~(\cite[Theorem~3.1]{hindman-thm}) states that for any finite partition of $\mathbb N,$ there exists an infinite (strictly increasing) sequence $(n_k)_{k\in\mathbb N}$ such that the set $\fs(n_k)_{k\in\mathbb N}$ is completely contained in a single cell of the partition.

We will use a slightly stronger form of the aforementioned theorem. Given two (strictly increasing) sequences of natural numbers $(n_k)_{k\in\mathbb N}$ and $(m_k)_{k\in\mathbb N}$, we say that the sequence 
$(m_k)_{k\in\mathbb N}$ is a {\em sum subsystem} of the sequence $(n_k)_{k\in\mathbb N}$ if there are finite subsets $F_1,F_2,\ldots,F_k,\ldots$ of $\mathbb N$ such that, for each  $k \in \mathbb N$, we have 
$\max(F_k)<\min(F_{k+1})$ and $m_k=\sum_{j\in F_k}n_j$. Note that this implies in particular that 
$\fs(m_k)_{k\in\mathbb N}\subseteq\fs(n_k)_{k\in\mathbb N}$. With this terminology, we record Hindman's theorem in the form that will be
used later:

\begin{theorem}[{\cite[Corollary~5.15]{hindman-strauss}}] \label{hindmanthm}
For every infinite \emph{(}strictly increasing\emph{)} sequence $(n_k)_{k\in\mathbb N}$ and for every finite partition $\{A_1,\ldots,A_m\}$ of the set $\fs(n_k)_{k\in\mathbb N}$, there exists an $i\in\{1,\ldots,m\}$ and a sum subsystem $(m_k)_{k\in\mathbb N}$ of the sequence $(n_k)_{k\in\mathbb N}$ such that $\fs(m_k)_{k\in\mathbb N}\subseteq A_i$.
\end{theorem}

Since $\mathbb N=\fs(2^{k-1})_{k\in\mathbb N}$, the original version of Hindman's theorem follows immediately from Theorem~\ref{hindmanthm} above.

\begin{theorem}\label{non-equi-implies-f-expanding}
Let $X$ be a finite tree and let $f \colon X\longrightarrow X$ be a map. If $f$ is not equicontinuous, then $X$ has an $f$-expanding arc.
\end{theorem}
\begin{proof}
Take an $x\in X$ such that $f$ is not equicontinuous at $x$. We have two cases.

\noindent {\bf Case 1:} The point $x$ is eventually periodic (i.e., there is $k \in \mathbb N$ such that $f^k(x)$ 
is a periodic point; equivalently, the set $\{f^n(x)\big|n\in\mathbb N\}$ is finite). This means that, replacing $x$ by some $f^k(x)$ if necessary (and using clause (1) of Lemma~\ref{lem:f-bad-fn-bad}), we may assume that $x$ is a 
periodic point, say with period $n$. Now use clause (2) of Lemma~\ref{lem:f-bad-fn-bad} to find an $i<n$ such
 that $f^n$ fails to be equicontinuous at $y=f^i(x)$, and notice that $y$ is a fixed point for $f^n$. If $X$ has an 
 $f^n$-expanding arc, then this is also an $f$-expanding arc and we are done. If, on the contrary, there are no 
 $f^n$-expanding arcs, then we can use Lemma~\ref{tree-to-dendrite} to obtain an $f^n$-invariant subspace $Y\subseteq X$ such that $Y$ is a  $k$-od for some $k\geq 2$, and $y \in \int_X(Y)$. By 
 Remark~\ref{non-equicontinuity-restriction}, the restricted map 
$f^n\upharpoonright Y \colon Y\longrightarrow Y$ fails to be equicontinuous at $y$, and so by 
Proposition~\ref{simple-f-expanding}, $Y$ must contain an $f^n$-expanding arc $I$. Then $I\subseteq X$ is also an 
$f$-expanding arc, and we are done. \vskip .2cm
 
\noindent {\bf Case 2:} The point $x$ is not eventually periodic. Then the set $\{f^n(x)\big|n\in\mathbb N\}$ is infinite. The space $X$ is a finite tree and so it can be written as a union of finitely many maximal free arcs $I_1,\ldots,I_t$
 in $X$ such that any two distinct $I_i,I_j$ have at most one (branching) point in common. Now we define subsets $A_0,A_1,\ldots,A_t$ of $\mathbb N$ as follows: $n\in A_0$ iff $f^n(x)$ is a branching point of $X$ and, for each $i\in\{1,\ldots,t\}$, $n\in A_i$ iff $f^n(x)\in I_i$ and $f^n(x)$ is not a branching point of $X.$ Clearly 
$$
A_i \cap A_j = \varnothing \hspace{.5cm} \mbox{for every } i,j\in \{0,1,\ldots,t\} \mbox{ with } i \ne j.
$$ 
Given $i \in \{0,1,\ldots,t\}$ the set $A_i$ can be empty. Since the set $\{f^n(x)\big|n\in\mathbb N\}$ is infinite, there exist distinct $m_1,m_2,\ldots,m_r \in \{0,1,\ldots t\}$ such that $A_{m_j} \neq\varnothing$ for every  
$j \in \{1,2,\ldots,r\}$ and $A_i = \varnothing$ for each $i \in \{0,1,\ldots,t\}\setminus \{m_1,m_2,\ldots,m_r\}$. Hence 
$\{A_{m_1},A_{m_2},\ldots,A_{m_r}\}$ is a finite partition of $\mathbb N$. Theorem~\ref{hindmanthm} provides us with a $j\in\{1,2,\ldots,r\}$ and an infinite strictly increasing sequence $n_1^1<\cdots<n_k^1<n_{k+1}^1<\cdots$ of natural numbers, such that the set $\fs(n_k^1)_{k\in\mathbb N}\subseteq A_{m_j}$. Since the points $f^n(x)$ as $n\in\mathbb N$ varies are pairwise distinct and $X$ only has finitely many branching points, $A_0$ is finite. Therefore $m_j\neq 0$ and so
\begin{equation*}
\{f^n(x)\big|n\in\fs(n_k^1)_{k\in\mathbb N}\}\subseteq I_{m_j}.
\end{equation*}
Use a homeomorphism of $[0,1]$ onto $I_{m_j}$ to equip $I_{m_j}$ with a linear order $\leq$. We now partition 
the set $\fs(n_k^1)_{k\geq 2}$ according to whether $f^n(x)<f^{n_1^1+n}(x)$ or $f^{n_1^1+n}(x)<f^n(x)$; a further application of Theorem~\ref{hindmanthm} allows us to obtain a sum subsystem $(n_k^2)_{k\geq 2}$ of 
$(n_k^1)_{k\geq 2}$ such that $\fs(n_k^2)_{k\geq 2}$ is contained in one piece of this partition. This means that 
there is an $R_1\in\{>,<\}$ such that, 
$$
f^{n_1^1+n}(x)\ R_1\ f^n(x), \hspace{.5cm} \mbox{for every } n\in\fs(n_k^2)_{k\geq 2}.
$$
Continuing this process by induction, we obtain, for each $K\in\mathbb N$, a sum subsystem $(n_k^{K+1})_{k\geq K+1}$ of $(n_k^K)_{k\geq K+1}$ and an $R_K\in\{>,<\}$ such that
$$
f^{n_K^K+n}(x)\ R_K\ f^n(x), \hspace{.5cm} \mbox{for each } n\in\fs(n_k^{K+1})_{k\geq K+1}.
$$
Now, an application of the pigeonhole principle allows us to obtain an infinite increasing sequence $(K_k)_{k\in\mathbb N}$ such that all the $R_{K_k}$ are equal, say, without loss of generality, to $>$. What this means is that, if we define the sequence $(n_k)_{k\in\mathbb N}$ by $n_k=n_{K_k}^{K_k}$, then for every $K\in\mathbb N$ and each $n\in\fs(n_k)_{k\geq K+1}$ we have $f^n(x)<f^{n_K+n}(x)$.

Now, for each  $K \in \mathbb N$ we define a point $y_K\in I_{m_j}$ by 
$$
y_K=\sup\{f^n(x)\big|n\in\fs(n_k)_{k\geq K}\}.
$$ 
Since $f^n(x)<f^{n_K+n}(x)$ for every $n\in\fs(n_k)_{k\geq K+1}$, we must have $y_{K+1}\leq y_K$ for every 
$K\in\mathbb N$. We may now break the proof into two further subcases (recall that $d$ is the metric on $X$).

\begin{description}
\item[Subcase 2.A] There is a $K\in\mathbb N$ such that $y_{K+1}=y_K$. Let $y=y_K=y_{K+1}$ and 
note that, for every $N\in\mathbb N$, there is an $l_N\in\fs(n_k)_{k\geq K+1}$ with 
$$
f^{l_N}(x)<y \hspace{.5cm} \mbox{and} \hspace{.5cm} d(f^{l_N}(x),y)<\frac{1}{N}.
$$ 
We have $f^{l_N}(x)<f^{n_K+l_N}(x)<y$; in particular, we also have 
$$
d(f^{n_K+l_N}(x),y)<\frac{1}{N}.
$$ 
It follows that 
$$
\lim_{N\to\infty}f^{l_N}(x)=y \hspace{.5cm} \mbox{and} \hspace{.5cm}
\lim_{N\to\infty}f^{n_K}(f^{l_N}(x))=y.
$$
By continuity of the function $f^{n_K}$, we may conclude that $f^{n_K}(y)=y$. Thus, all the points 
$$
y,f(y),\ldots,f^{n_K-1}(y)
$$ 
\noindent are fixed points of the map $f^{n_K}$. If $X$ contains an $f^{n_K}$-expanding arc, then this arc is 
also $f$-expanding and we are done, so assume otherwise. Then we may apply Lemma~\ref{tree-to-dendrite} to each of the points $y,f(y),\ldots,f^{n_K-1}(y)$ to get $f^{n_K}$-invariant subcontinua 
$$
Y_0,Y_1,\ldots,Y_{n_K-1}\subseteq X
$$ 
\noindent such that, for every $i \in \{0,1,\ldots,n_K-1\},$ we have $f^i(y) \in \int_X(Y_i)$ and each $Y_i$ is 
a $k_i$-od for some $k_i \geq 2$. Let $\varepsilon>0$ be such that for every $i<n_K$, the ball 
centered at 
$f^i(y)$ with radius $\varepsilon$ is contained in $Y_i.$ By the continuity of the functions $f,f^2,\ldots,f^{n_K-1}$ 
we get a  $\delta > 0$ such that, if $d(z,y)<\delta$, then $d(f^i(z),f^i(y))<\varepsilon.$ Hence, for each $i<n_K$, if $d(z,y)<\delta$ then $f^i(z) \in \int_X(Y_i)$. Note that $y$ contains points of the form $f^n(x)$ arbitrarily close, and 
all the points of the form $f^n(x)$ are points where the map $f$ is not equicontinuous (by clause (1) of
Lemma~\ref{lem:f-bad-fn-bad}). Hence we can find a $z$ with $d(z,y)<\delta$ such that $f$ is not equicontinuous at $z$; now use clause (2) of Lemma~\ref{lem:f-bad-fn-bad} to get $i<n_K$ such that $f^{n_K}$ 
is not equicontinuous at $f^i(z) \in \int_X(Y_i)$. Since $Y_i$ is $f^{n_K}$-invariant, we may conclude that 
$f^{n_K}\upharpoonright Y_i$ is not an equicontinuous map (see Remark~\ref{non-equicontinuity-restriction}). Since $Y_i$ is a $k_i$-od, by Proposition~\ref{simple-f-expanding}, the subcontinuum $Y_i$ of $X$ must have an $f^{n_K}$-expanding arc, and we are done.\vskip .2cm

\item[Subcase 2.B] $y_{K+1}<y_K$ for every $K\in\mathbb N.$ Then let $y=\inf\{y_K\big|K\in\mathbb N\}$. For each $K\in\mathbb N$ fix an $m_K\in\fs(n_k)_{k\geq K}$ such that $y_{K+1}<f^{m_K}(x)<y_K$. If for some $K\in\mathbb N$, it is not the case that $y<f^{n_K}(y)$, then we must have 
$$
y f^{m_{K+1}}(y)\subsetneq f^{n_K}(y) f^{n_K+m_{K+1}}(x)\subseteq f^{n_K}[y f^{m_{K+1}}(x)]
$$
\noindent  and therefore there is an $f$-expanding arc and we are done, so assume that for all $K\in\mathbb N$ we have $y<f^{n_K}(y)$. The points $f^{m_k}(x)$ for $k>K$ are arbitrarily close to $y$ and they all satisfy $f^{n_K}(f^{m_k}(x))=f^{n_K+m_k}(x)\leq y_K$, so by continuity of $f^{n_K}$ we have $f^{n_K}(y)\leq y_K$.

We define connected subspaces $Y_1,Y_2\subseteq X$ as follows. $Y_1$ is the connected component of 
$X\setminus\{f^{m_2}(x)\}$ that does not contain $y$, and $Y_2$ is the connected component of 
$X\setminus\{f^{n_3}(y)\}$ containing $y$. Since 
$$
y<f^{n_3}(y)\leq y_3<f^{m_2}(x)<f^{n_1+m_2}(x),
$$ 
\noindent and all such points belong to the maximal free arc $I_{m_j}$ in $X$, we can write 
$$
X=Y_1\cup f^{n_3}(y) f^{m_2}(x)\cup Y_2,
$$ 
\noindent and the union is disjoint. Since $f^{n_1+m_2}(x)\in Y_1$, by Proposition~\ref{fixpoint} there is a 
$z_1\in Y_1\cap\fix(f^{n_1})$; now if we let $K$ be sufficiently large that $y_K<f^{n_3}(y)$ then we will have 
$f^{n_K}(y)\in Y_2$ and so again by Proposition~\ref{fixpoint} there exists a $z_2\in Y_2\cap\fix(f^{n_K-n_3})$. 
Letting $N=(n_K-n_3)n_1$, we get that $z_1,z_2\in\fix(f^N)$, and $f^{m_2}(x)\in z_1 z_2\setminus\fix(f^N)$. 
Hence $\fix(f^N)$ is a disconnected set, and so by Lemma~\ref{arc-iff-disconnection}, $X$ must have an 
$f$-expanding arc.
\end{description}
\end{proof}

\section{The Ellis remainder and ultrafilter-limits}\label{sect:ultrafilters}

In this section we introduce the notion of ultrafilter-limits and point out the relation of this concept with that of 
the Ellis remainder, with the objective of establishing the equivalence of items (a), (e), (h) and (i) from Theorem~\ref{mainthm}. 

\begin{definition}\hfill
\begin{enumerate}
\item An {\bf ultrafilter} on $\mathbb N$ is a family $u$ of subsets of $\mathbb N$ such that
\begin{enumerate}
\item $u$ is nonempty and $\varnothing\notin u$;
\item if $A,B\in u,$ then $A\cap B\in u$;
\item if $A\in u$ and $A\subseteq B\subseteq \mathbb N,$ then $B\in u$;
\item whenever $\mathbb N=A\cup B$, then either $A\in u$ or $B\in u$.
\end{enumerate}
\item An ultrafilter $u$ on $\mathbb N$ is {\bf principal} if there exists an $n\in\mathbb N$ such that 
        $u=\{A\subseteq\mathbb N\big|n\in A\}$; otherwise we say that $u$ is {\bf nonprincipal}.
\item We use the symbol $\beta\mathbb N$ to denote the set of all ultrafilters on $\mathbb N$, and we denote 
        with $\mathbb N^*$ the set of all nonprincipal ultrafilters on $\mathbb N$.
\item Given a metric space $(X,d)$, a sequence $(x_n)_{n\in\mathbb N}$ of points on $X$, and an ultrafilter 
         $u\in\beta\mathbb N$, we say that $x$ is the $u$ {\bf ultrafilter-limit} of $(x_n)_{n\in\mathbb N},$ in
         symbols $x=u\text{-lim}_{n\to\infty}x_n,$ if for every $\varepsilon>0$ the set 
         $\{n\in\mathbb N\big|d(x,x_n)<\varepsilon\}\in u$.
\item Given a metric space $X$, a function $f \colon X\longrightarrow X$, and an ultrafilter $u\in\beta\mathbb N$, 
         we define the $u$ {\bf ultrafilter-limit function} $f^u\colon X\longrightarrow X$ (also called the $u$-th iterate 
         of $f$) by $f^u(x)=u\text{-lim}_{n\to\infty}f^n(x)$.
\end{enumerate}
\end{definition}

Given a compact metric space $X,$ a map $f \colon X\longrightarrow X$ and $x \in X$, 
it can be shown that 
$$
\omega(x,f)=\{f^u(x)\big|u\in\mathbb N^*\}.
$$
A few comments about the above definitions are in order. For each $n\in\mathbb N$, it is common to identify the natural number $n$ with the principal ultrafilter $u_n=\{A\subseteq\mathbb N\big|n\in A\}$; this way we can think 
of $\mathbb N$ as a subset of $\beta\mathbb N$, and we have $\mathbb N^*=\beta\mathbb N\setminus\mathbb N$. Furthermore, one can topologize $\beta\mathbb N$ by declaring the sets 
$\bar{A}=\{u\in \beta \mathbb N\big|A\in u\}$ to be open, for each $A\subseteq \mathbb N$; this endows 
$\beta\mathbb N$ with a compact Hausdorff topology containing $\mathbb N$ as a discrete dense subspace (\cite[Lemma~3.17 and Theorems~3.18 and 3.28]{hindman-strauss}). Regarding the concept of a 
$u$-limit, it is worth pointing out that, in a compact metric space $X$, every sequence $(x_n)_{n\in\mathbb N}$ of points will have a unique $u$-limit (for every $u\in\beta\mathbb N$)(\cite[Theorem~3.48]{hindman-strauss}). Moreover, if $u_n$ is the principal ultrafilter $\{A\subseteq\mathbb N\big|n\in A\}$, then 
$u_n\text{-lim}_{m\to\infty}x_m=x_n$; similarly (and as a consequence of the above), for a function 
$f \colon X\longrightarrow X$ we will have that $f^{u_n}=f^n$. Thus, no confusion should arise if we sometimes 
abuse notation and write $n$ instead of $u_n$.

Furthermore, it is possible to equip $\beta\mathbb N$ with a right-topological semigroup operation, denoted by 
$+$. That is, $+$ is an associative binary operation on $\beta\mathbb N$ such that, for each fixed 
$u\in \beta \mathbb N$, the function $v\longmapsto u+v$ is continuous. The operation is given by the formula
\begin{equation*}
u+v=\{A\subseteq\mathbb N\big|\{n\in\mathbb N\big|\{m\in\mathbb N\big|n+m\in A\}\in v\}\in u\}.
\end{equation*}
This operation extends the usual sum on $\mathbb N$, in the sense that, if $n,m\in\mathbb N$ and $u_n,u_m$ are the corresponding principal ultrafilters, then $u_n+u_m=u_{n+m}$, although $+$ is not commutative on all of $\beta\mathbb N$. It is possible to verify that, for any $u,v\in\beta\mathbb N$, we have $f^u\circ f^v=f^{u+v}$ 
(see~\cite[p. 38]{blass-ultra-dynamics}).

As we mentioned in the Introduction, the equation 
$$
E(X,f)=\{f^u\big|u\in\beta\mathbb N\},
$$
\noindent shown in~\cite[Theorem~2.2]{garcia-sanchis} and which holds for every map 
$f \colon X\longrightarrow X$ on a compact metric space $X$, is the main reason why obtaining information about 
the ultrafilter-limit functions $f^u$ has a great deal of importance within the study of the discrete dynamical system $(X,f)$. 
At this moment, we aim to prove that the existence of expanding arcs implies the discontinuity of ultrafilter-limit functions. We begin by introducing a definition that will help to expedite the statement of the subsequent lemmas.

\begin{definition}
Let $X$ be a metric space.
\begin{enumerate}
\item Let $I\subseteq X$ be an arc, and let $(x_n)_{n\in\mathbb N}$ be a sequence of elements of $I$. We say 
         that the sequence is {\bf $I$-monotone} if it is monotone (i.e., either increasing or decreasing) when viewed 
         as a sequence on the unit interval $[0,1]$ via a homeomorphism $:I\longrightarrow[0,1]$. Equivalently, the 
         sequence $(x_n)_{n\in\mathbb N}$ is monotone if $x_{n+1}\in x_n x_{n+2}$ for each $n\in\mathbb N$ 
         (noting that $x_n x_{n+2}\subseteq I$).
\item If $g \colon X\longrightarrow X$ is a map, a sequence $(x_n)_{n\in\mathbb N}$ of elements of 
         some arc $I\subseteq X$ is said to be {\bf $g$-backwards} if it is $I$-monotone and for each $n \in \mathbb N$ 
         we have $g(x_{n+1})=x_n$.
\end{enumerate}
\end{definition}

\begin{remark}
Note that, by compactness of an arc and monotonicity of backward sequences, any $g$-backward sequence on a dendrite is always convergent. Furthermore, the limit of the sequence is a fixed point of $g$.
\end{remark}

\begin{lemma}\label{expandingtoadequate}
Let $X$ be a dendrite and let $f \colon X\longrightarrow X$ be a map, and suppose that there is an $f$-expanding arc $I\subseteq X$. Then the following two conditions hold:
\begin{enumerate}
\item[\emph{(1)}] for some $m\in\mathbb N$ there exists an $f^m$-backward sequence $(y_n)_{n\in\mathbb N}$ 
        in $I$;
\item[\emph{(2)}] the set $\per(f)$ is disconnected.
\end{enumerate}
\end{lemma}

\begin{proof}
If there is an $f$-expanding arc $I$ in $X$ then, by Lemma~\ref{arc-iff-disconnection}, there exist points $x,y\in X$ and  $n\in\mathbb N$ such that $f^n(x)=x$, $f^n(y)\neq y$, and $y\in x f^n(y)\subseteq f^n[xy]$, so we can find a $y_1\in xy\setminus\{y\}$ such that $f^n(y_1)=y$. Now $y_1\in xy=f^n(x)f^n(y_1)\subseteq f^n[xy_1]$, so there exists a $y_2\in xy_1\setminus\{y_1\}$ such that $f^n(y_2)=y_1$. Continuing by induction, if we already know $y_1,\ldots,y_k$ with $f^n(y_i)=y_{i-1}$ and $y_i\in xy_{i-1}\setminus\{y_{i-1}\}$, then 
$$
y_k\in xy_{k-1}=f^n(x)f^n(y_k)\subseteq f^n[xy_k],
$$ 
and so there exists a $y_{k+1}\in xy_k\setminus\{y_k\}$ such that $f^n(y_{k+1})=y_k$. This way we obtain a sequence $(y_n)_{n\in\mathbb N}$ which is $f^n$-backward. So (1) holds.

To show (2), we use the points $x,y$ and the sequence $(y_n)_{n\in\mathbb N}$ obtained in (1). Since 
$y\in x f^n(y)$, by Proposition~\ref{fixpoint} there is a point $z\in\fix(f^n)$ such that $y\in xz$; then we have 
$x,z\in\per(f)$, so it suffices to show that $xz\setminus\per(f)\neq\varnothing$. If $y\notin\per(f)$ we are done, so assume that $y=\per(f)$, say with period $k$. Then $\{f^n(y)\big|n\in\mathbb N\}=\{y,f(y),\ldots,f^{k-1}(y)\}$; since 
the $y_n$ are pairwise distinct we can choose an $n\in\mathbb N$ such that $y_n\notin\{y,f(y),\ldots,f^{k-1}(y)\}$. Then 
$$
f^m(y_n)\in\{y_{n-1},\ldots,y_1,y,f(y),\ldots,f^{k-1}(y)\} \hspace{.5cm} \mbox{for all } m\in\mathbb N,
$$ 
\noindent thus for each $m\in\mathbb N$ we have $f^m(y_n)\neq y_n$ and so $y_n\notin\per(f)$. Since $y_n\in xz$, 
the proof is finished.
\end{proof}

\begin{corollary}\label{connectednessofper}
For a map $f \colon X\longrightarrow X$ with $X$ a dendrite, the following are equivalent:
\begin{enumerate}
\item[\emph{(e)}] there is no $f$-expanding arc in $X$;
\item[\emph{(g)}]  the set $\per(f)$ is connected.
\end{enumerate}
\end{corollary}

\begin{proof}
Suppose that $\per(f)$ is disconnected, and find $x,y\in\per(f)$ such that there exists a $z\in xy\setminus\per(f)$. If $x$ has period $n$ and $y$ has period $m$, then we have $x,y\in\fix(f^{nm})$; as $z$ is not periodic, we have
 $z\in xy\setminus\fix(f^{nm})$. Hence the set $\fix(f^{nm})$ is disconnected and so, by 
 Lemma~\ref{arc-iff-disconnection}, $X$ contains an $f$-expanding arc. Conversely, if $X$ contains an 
 $f$-expanding arc, use Lemma~\ref{expandingtoadequate}.
\end{proof}

Now, in order to use $g$-backward sequences to deduce discontinuity of elements in the Ellis remainder, we will introduce a fairly stronger definition that allows us to work in a slightly more general context. In what follows, it will be convenient that the indexing of our sequences starts at 0 rather than at 1.

\begin{definition}
Let $X$ be a compact metric space, and let $g \colon X\longrightarrow X$ be a map. A sequence 
$(x_n)_{n\in\mathbb N\cup\{0\}}$ of elements of $X$ will be said to be {\bf $g$-divergent} if the following three conditions hold:
\begin{enumerate}
\item $x=\lim_{n\to\infty}x_n$ exists in $X$;
\item for each $n\in\mathbb N$, $g(x_{n+1})=x_n$ (this implies that $g(x)=x$);
\item there exists an open neighbourhood $U\subseteq X$ of $x$ such that $U\cap\{g^n(x_0)\big|n\in\mathbb N\}
         =\varnothing$.
\end{enumerate}
\end{definition}

It is not hard to see that $g$-divergent sequences can only exist if $g$ fails to be equicontinuous. As a matter of 
fact, much more is true, as seen in the following theorem.

\begin{theorem}\label{discontinuity}
Let $X$ be an arbitrary compact metric space and let $g \colon X\longrightarrow X$ be a map. If 
there is an $m\in\mathbb N$ such that $X$ contains a $g^m$-divergent sequence, then for every nonprincipal ultrafilter $u\in\mathbb N^*$, the function $g^u$ is discontinuous.
\end{theorem}

\begin{proof}
Let $(x_n)_{n\in\mathbb N\cup\{0\}}$ be the hypothesized $g^m$-divergent sequence, let $x=\lim_{k\to\infty}x_k$, and let $U$ be an open set containing $x$ such that $U\cap\{g^{mn}(x_0)\big|n\in\mathbb N\}=\varnothing$.

Now let $u\in\mathbb N^*$ be an arbitrary nonprincipal ultrafilter. There exists a unique $0 \leq i<m$ such that 
$m\mathbb N+i\in u$, so that $m\mathbb N\in u-i$. This means that it makes sense to consider the Rudin--Keisler image $v$ of the ultrafilter $u-i$ under the mapping $\colon m\mathbb N\longrightarrow\mathbb N$ given by $mk\longmapsto k$. So we have that $mv+i=u$ (where $mv$ denotes the Rudin--Keisler image of the ultrafilter 
$v$ under the mapping $k\longmapsto mk$).

Define a new sequence $(y_n)_{n\in\mathbb N\cup\{0\}}$ by letting $y_n=g^{m-i}(x_n)$, and let $y=g^{m-i}(x)$. 
Since $(x_n)_{n\in\mathbb N\cup\{0\}}$ converges to $x$ and $g^{m-i}$ is continuous, $(y_n)_{\mathbb N\cup\{0\}}$ will converge to $y$. We now proceed to observe that 
\begin{eqnarray*}
g^u(y) & = & g^{mv+i}(g^{m-i}(x))=(g^m)^v(g^i(g^{m-i}(x))) \\
 & = & (g^m)^v(g^m(x))=(g^m)^v(x)=x\in U,
\end{eqnarray*}
and, for each $k\in\mathbb N$, we have
\begin{equation*}
g^u(y_k)=g^{mv+i}(g^{m-i}(x_k))=(g^m)^v(g^m(x_k))=g^{mv}(x_{k-1}).
\end{equation*}
By definition of ultrafilter-limits, $g^{mv}(x_{k-1})$ must be an accumulation point of the set 
$\{g^{mn}(x_{k-1})\big|n\in\mathbb N\}$. However, for $n>k-1$ we have 
$g^{mn}(x_{k-1})=g^{m(n-k+1)}(x_0)\notin U$, so $g^u(y_k)\notin U$ for every $k\in\mathbb N$, and therefore the sequence $(g^u(y_k))_{k\in\mathbb N}$ does not converge to $x=g^u(y)$, showing that the function $g^u$ is discontinuous at $y$, and we are done.
\end{proof}

The previous lemma works for every compact metric space. For certain dendrites, there is a relation between $g$-backwards sequences and $g$-divergent sequences.

\begin{lemma}\label{adequatetoveryadequate}
Let $X$ be a dendrite with only finitely many branching points, and let $g \colon X\longrightarrow X$ be a map. If there is an arc $I\subseteq X$ such that $I$ contains a $g$-backwards sequence, then there exists an $m\in\mathbb N$ such that $X$ has a $g^m$-divergent sequence.
\end{lemma}

\begin{proof}
Let $(x_n)_{n\in\mathbb N\cup\{0\}}$ be a $g$-backwards sequence in the arc $I$, and let $x=\lim_{n\to\infty}x_n$. Notice that $x$ is a fixed point of $g$, and therefore $\lim_{k\to\infty}g^k(x)=x$.

Now let us fix some notation. First of all, since $X$ has only finitely many branching points, we may shrink $I$ and drop finitely many terms of the sequence (and shift indices afterwards so that our sequence indexing still starts at 0) $(x_n)_{n\in\mathbb N\cup\{0\}}$ to ensure that $I$ is a free arc in $X$. Now linearly order the arc $I$ via a homeomorphism with $[0,1]$ in such a way that $x_0<x$. Then the monotonicity of the $g$-backwards sequence $(x_n)_{n\in\mathbb N\cup\{0\}}$ means in this case that the sequence is increasing. Now let $r_I \colon X\longrightarrow I$ be the first point function from $X$ onto the subcontinuum $I$ of $X$. We will analyze the $g$-orbit of $x_0.$ There are two cases to consider. \vskip .2cm

\noindent {\bf Case 1:} For every $m\in\mathbb N$, $r_I(g^m(x_0))\leq x_0$. In this case, for each fixed 
$n\in\mathbb N\cup\{0\}$ we have that $g^{m+n}(x_n)=g^m(x_0)$, and so $r_I(g^{m+n}(x_n))\leq x_0<x$ for every $m\in\mathbb N$. Since $X$ is a dendrite and hence uniformly locally arcwise connected, there must be a 
$\delta>0$ such that, whenever $d(x,z)<\delta$ and $x \ne z$, then the arc $xz$ must have diameter smaller than that of the arc 
$x_0 x$. In particular, if $r_I(z)\leq x_0$, then $d(x,z)\geq\delta$. So if we let $U$ be the ball centred at $x$ with radius $\delta$, then for every $n\in\mathbb N$ it is the case that $g^n(x_0)\notin U$, and consequently the sequence $(x_n)_{n\in\mathbb N}$ itself is already $g$-divergent. \vskip .2cm

\noindent {\bf Case 2:}
There exists an $m\in\mathbb N$ such that $x_0\leq r_I(g^m(x_0))$. Fix one such $m$, and notice that the 
function $r_I\circ(g^m\upharpoonright I):I\longrightarrow I$ satisfies 
$$
r_I(g^m(x_m))=r_I(x_0)=x_0\leq x_m \hspace{.5cm} \mbox{and} \hspace{.5cm} 
 x_0\leq r_I(g^m(x_0)).
 $$
Therefore (by a standard result for maps in the unit interval) this map 
must have a fixed point in $x_0 x_m$, that is, there is a $z_0\in x_0 x_m$ with $z_0=r_I(g^m(z_0))$. Since $I$ is 
a free arc in $X$, we have that $r_I(w)$ is one of the endpoints of $I$ whenever $w\notin I$. Since 
$z_0\in x_0 x_m\setminus\{x_0,x_m\}$ (so $z_0$ is an interior point of $I$), from $z_0=r_I(g^m(z_0))$ it follows 
that $z_0=g^m(z_0)$ and so $z_0$ is actually a fixed point of the map $g^m$.

Now $x_0 x_m=g^m(x_m) g^m(x_{2m})\subseteq g^m[x_m x_{2m}]$, so there must exist a $z_1\in x_m x_{2m}$ such that $g^m(z_1)=z_0$. We continue this process by induction: given a 
$$
z_n\in x_{nm} x_{(n+1)m}=g^m(x_{(n+1)m})g^m(x_{(n+2)m})\subseteq g^m[x_{(n+1)m} x_{(n+2)m}],
$$
\noindent we find a $z_{n+1}\in x_{(n+1)m} x_{(n+2)m}$ such that $g^m(z_{n+1})=z_n$. This way we obtain a monotone sequence $(z_n)_{n\in\mathbb N\cup\{0\}}$, with limit $x$, which is $g^m$-backwards and where $z_0\in\fix(g^m)$. Since $z_0\neq x$, any open set $U$ containing $x$ and not containing $z_0$ will satisfy $(g^m)^n(z_0)=z_0\notin U$, for every $n\in\mathbb N$. Therefore the sequence 
$(z_n)_{n\in\mathbb N\cup\{0\}}$ is $g^m$-divergent.
\end{proof}

We are ready to prove the Main Theorem of this paper.

\begin{proof}[Proof of Main Theorem~\ref{mainthm} \emph{(}and of clause \emph{(1)} of Remark~\ref{aftermainthm}\emph{)}]
The equivalence of (a), (b), (c) and (d) is established in Proposition~\ref{cor:firstfour}. The equivalence of (e) and 
(f) is Lemma~\ref{arc-iff-disconnection}, and that of (e) and (g) is Corollary~\ref{connectednessofper}; in both cases this equivalence works for arbitrary dendrites. Finally, (e) implies (a) by Theorem~\ref{non-equi-implies-f-expanding}; (a) implies (h) easily (by the remark in the Introduction right after Definition~\ref{def:ellis}), and it is obvious that (h) implies (i) and that (d) implies (g). We also have that (i) implies (e): by contrapositive, if there exists an $f$-expanding arc in $X$ then there is an $f^m$-backward sequence for some $m$, by Lemma~\ref{expandingtoadequate}; this yields an $n\in\mathbb N$ such that there is an 
$f^{mn}$-divergent sequence by Lemma~\ref{adequatetoveryadequate}, and this in turn implies that there is no $u\in\mathbb N^*$ such that $f^u$ is continuous, by Theorem~\ref{discontinuity}. The last chain of implications establishes the equivalence of (a) with (e), (h) and (i), which finishes the proof.
\end{proof}

\section{Examples and open problems}

This section contains examples showing that the previous results cannot be extended to other kinds of dendrites. Theorem~\ref{mainthm} holds for finite trees, and trees are dendrites satisfying two additional conditions: that they have finitely many branching points, and that each branching point has finite order. We show examples of dendrites where one of these two conditions fails. Afterwards, we finish the paper by making a few observations about functions defined on finite graphs.

\subsection{Dendrites with finitely many branching points}

In this subsection we proceed to exhibit an example of a dendrite, and two maps defined on it, which together show that none of the equivalences between (a) and (b), (c), or any of (e)-(i) from Theorem~\ref{mainthm} are generalizable to dendrites with finitely many branching points (meaning that the hypothesis that all branching points are of finite order is really necessary in Theorem~\ref{mainthm}).

\begin{remark}\label{rem:aiffd}
A few words regarding the equivalence between (a) and (d) are in order. By~\cite[Theorem~4.12]{camargo-rincon-uzcategui} together with~\cite[Lemma~2.6]{four-chinese-authors}, the implication from (a) to (d) still holds if $X$ is merely a dendrite with finitely many branching points. The reverse implication holds for an arbitrary dendrite and a {\em surjective} map by~\cite[Theorem~5.2]{mai}. Surjectivity, however, is necessary: Sun et. 
al.~\cite[Example~2.9]{sun2014} exhibit an example of a dendrite $X$ and a map 
$f:X\longrightarrow X$ such that $f$ fails to be equicontinuous (although 
$f\upharpoonright\bigcap_{n=1}^\infty f^n[X]$ {\em is} equicontinuous) yet $\per(f)=\bigcap_{n=1}^\infty f^n[X]$. 
The dendrite in this example has infinitely many branching points; we do not know of an example of this 
phenomenon with a dendrite that has finitely many branching points (cf. Question~\ref{question:aiffd}).
\end{remark}

\begin{example}\label{finitely-many-branching-points}
{\em A dendrite $X$ with a unique branching point, which has infinite order, and maps 
$f,g \colon X\longrightarrow X$ such that $f$ satisfies all conditions from \emph{(e)} through \emph{(i)} of \emph{Theorem~\ref{mainthm}} but fails to be equicontinuous, while $g$ is equicontinuous but fails to satisfy conditions \emph{(b)} and \emph{(c)} of \emph{Theorem~\ref{mainthm}}.} \vskip .2cm

For other purposes, the dendrite $X$ together with the map $f$, appear 
in~\cite[Example 5.1]{camargo-rincon-uzcategui}. We reproduce their description here for three reasons: for the reader's convenience, to point out a few observations about the map $f$ that are not made 
in~\cite{camargo-rincon-uzcategui}, and in order to be able to also describe the map $g$. We build $X$ by taking infinitely many disjoint arcs indexed by $\mathbb Z$, $\{I_n\big|n\in\mathbb Z\}$, with each $I_n$ of length 
$\frac{1}{2^{|n|}}$, and identifying in a single point $v$ (the {\em vertex}) one end of each $I_n$. The result 
$X = \bigcup_{n \in \mathbb Z}I_n$ is a dendrite with a single infinite-order branching point $v$, as in Figure~\ref{fig:infinite-order}.

\begin{figure}[t]
\centering
\includegraphics[scale=0.6]{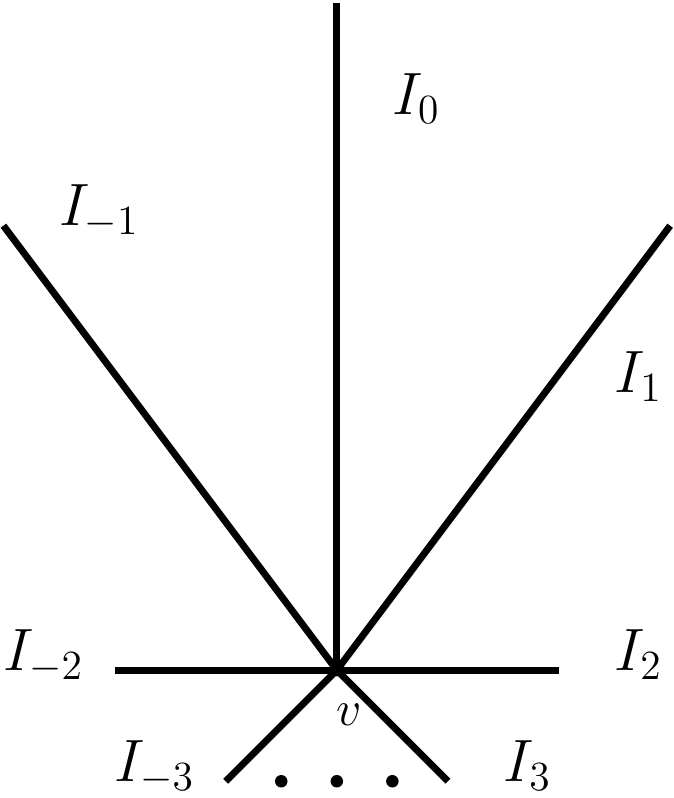}
\caption{The dendrite $X$ that has $v$ as its only branching point of infinite order.}
\label{fig:infinite-order}
\end{figure}

Now, for each $n\in\mathbb Z$, let us consider a map $h_n \colon I_n\longrightarrow I_{n+1}$ defined by fixing $v$, and, for each $e \in(0,\frac{1}{2^{|n|}}]$, if $x$ is the unique element of $I_n \setminus \{v\}$ at distance $e$ from $v$, then $h_n(x)$ is the unique element of $I_{n+1} \setminus \{v\}$ at distance $2^{-\frac{n}{|n|}}e$ (at distance 
$\frac{e}{2}$ in the case $n=0$) from $v$. Note that $h_n$ maps $I_n$ homeomorphically onto $I_{n+1}$.

We let $f \colon X\longrightarrow X$ be defined by $f=\bigcup_{n\in\mathbb Z}h_n$, that is, by $f(v)=v$ and $f(x)=h_n(x)$ whenever $n$ is the unique element in $\mathbb Z$ so that $x\in I_n \setminus \{v\}$. The sequence 
$(x_k)_{k\in\mathbb N}$, where $x_k$ is the endpoint of $I_{-k}$ that is distinct from $v$ (this sequence converges to 
$v$), together with $\varepsilon= \frac{1}{2}$ and the sequence of indices $(n_k)_{k\in\mathbb N}$ given by 
$n_k=k$, witness the failure of the equicontinuity of $f$ at $v$ (since $f^{n_k}(x_k)=x_0$, where $x_0$ is the endpoint of $I_0$ that is distinct from $v$). For each $n\in\mathbb N$, we have 
\begin{equation}\label{p1}
\fix(f^n)=\per(f)=\{v\},
\end{equation}
\noindent which is a connected set. Using (\ref{p1}) it is straightforward to see that property $(j^\prime)$ of
Lemma~\ref{arc-iff-disconnection} is not satisfied. Hence, by the same lemma, $X$ has no $f$-expanding arcs. 
Note that
$$
\omega(v,f) = \{v\} \subsetneq X =  \bigcap_{m=1}^\infty f^m[X].
$$
Moreover, for every $x\in X$ we have $\lim_{n\to\infty}f^n(x)=v$, which implies that, for each nonprincipal ultrafilter $u\in\mathbb N^*$, $f^u \colon X\longrightarrow X$ is the map with constant value $v$, which is continuous.

We now proceed to describe the map $g$. We stipulate that $g\upharpoonright\bigcup_{n\leq 2}I_n$ is the 
identity map. For each positive $m\in\mathbb N\setminus\{2^n\big|n\in\mathbb N\cup\{0\}\}$, we let $g\upharpoonright I_m=h_m$; finally, we let 
$$
g\upharpoonright I_m=h_{m-(2^{n-1}-1)}^{-1}\circ h_{m-(2^{n-1}-2)}^{-1}\circ\cdots\circ h_{m-1}^{-1},
\hspace{.2cm} \mbox{whenever } m=2^n \mbox{ with } n\geq 2.
$$ 
Hence we have $g\upharpoonright I_{2^n}:I_{2^n}\longrightarrow I_{2^{n-1}+1}$. Therefore, for every $n\geq 2$, 
the map $g$ will cyclically permute the finite sequence of arcs $(I_{2^{n-1}+1},I_{2^{n-1}+2},\ldots,I_{2^n})$, in such a way that $f^{2^{n-1}}\upharpoonright\bigcup_{i=2^{n-1}+1}^{2^n}I_i$ is the identity map (and $g$ will fix every point in each of the $I_m$ for $m\in\mathbb Z$ with $m\leq 2$). As a result of this, we will have that 
$\per(g)=X$, and so $g$ will be equicontinuous by~\cite[Theorem~4.14]{camargo-rincon-uzcategui}; at the same 
time, although $f$ is pointwise-periodic, $X$ contains points of arbitrarily high period (if $x\in I_m$ 
for $2^{n-1}+1\leq m\leq 2^n$, $n\in\mathbb N\setminus\{1\}$, then the period of $x$ is equal to $2^{n-1}$) and therefore, for every $n\in\mathbb N$, we have $\fix(g^n)\neq\bigcap_{m=1}^\infty g^m[X]$ and $g^n\upharpoonright\bigcap_{m=1}^\infty g^m[X]$ is not the identity map.
\end{example}

\subsection{Dendrites with branching points of finite order}

We now show that, if we drop the requirement that the dendrite $X$ has finitely many branching points, then none 
of the equivalences of equicontinuity from Theorem~\ref{mainthm} holds. The first few equivalences can be seen to fail by looking at~\cite[Example 5.4]{camargo-rincon-uzcategui}, which is the Gehman dendrite $X$ (this dendrite is described and pictured in~\cite[Example~10.39]{nadler}) with all branching points of finite order (with infinitely many branching points), and a surjective equicontinuous map $f \colon X\longrightarrow X$ such that $\per(f)\neq X$ (consequently, $X=\bigcap_{m=1}^\infty f^m[X]$ and $\per(f)\neq\bigcap_{m=1}^\infty f^m[X]$). Therefore $f$ is equicontinuous but fails to satisfy conditions (b), (c) and (d) of Theorem~\ref{mainthm}.

Now for the remaining equivalences, the following example finishes our analysis.

\begin{example}\label{many-branching-non-equicontinuous}
{\em A dendrite $X$ with infinitely many branching points \emph{(}each of which has finite order\emph{)} and a map $f \colon X\longrightarrow X$ that fails to be equicontinuous, but satisfies \emph{(e)} through 
\emph{(i)} of \emph{Theorem~\ref{mainthm}.}} \vskip .2cm

We build $X$ as a subset of $\mathbb R^2$ as follows. We let $K=[0,1]\times\{0\}$ and, for each 
$n\in\mathbb N\cup\{0\}$, we let $I_n=\left\{\frac{1}{2^n}\right\}\times\left[0,\frac{1}{2^n}\right]$ and 
$J_n=\left\{\frac{1}{2^n}\right\}\times\left[-\frac{1}{2^n},0\right]$. Define
\begin{equation*}
X=K\cup\left(\bigcup_{n=0}^\infty I_n\right)\cup\left(\bigcup_{n=0}^\infty J_n\right).
\end{equation*}
For notational convenience, we write $K=\bigcup_{n=1}^\infty K_n$ where $K_n=
\left[\frac{1}{2^n},\frac{1}{2^{n-1}}\right] \times \{0\}$ for each $n\in\mathbb N$. Now we define the map $f \colon X\longrightarrow X$ as follows. First make $f\upharpoonright K$ the identity map. Now, for 
each $n\in\mathbb N$, $f\upharpoonright I_n$ is given as follows: 
for $\left(\frac{1}{2^n},y\right)\in I_n$ ($0\leq y\leq\frac{1}{2^n}$), we let

\begin{equation*}
f\left(\frac{1}{2^n},y\right)= \left\{
\begin{array}{ll}
\left(\frac{1}{2^n}+2y,0\right), &\mbox{if } 0\leq y\leq\frac{1}{2^{n+1}}; \medskip\\
\left(\frac{1}{2^{n-1}},4\left(y-\frac{1}{2^{n+1}}\right)\right), &\mbox{if } \frac{1}{2^{n+1}}\leq y\leq\frac{1}{2^n},
\end{array}\right.
\end{equation*}

\noindent so that $f$ maps $I_n$ homeomorphically onto $K_n\cup I_{n-1}$. Furthermore, $f\upharpoonright I_0$ 
is defined by letting $f(0,y)=(0,-y)$ so that $f$ maps $I_0$ homeomorphically onto $J_0$. Finally, for each $n\in\mathbb N\cup\{0\}$, we define $f\upharpoonright J_n$ by letting

\begin{equation*}
f\left(\frac{1}{2^n},y\right)= \left\{
\begin{array}{ll}
\left(\frac{1}{2^n}+y,0\right), &\mbox{if } -\frac{1}{2^{n+1}}\leq y\leq 0;  \medskip\\
\left(\frac{1}{2^{n+1}},y+\frac{1}{2^{n+1}}\right), &\mbox{if } -\frac{1}{2^n}\leq y\leq-\frac{1}{2^{n+1}},
\end{array}\right.
\end{equation*}

\noindent whenever $\left(\frac{1}{2^n},y\right)\in J_n$ ($-\frac{1}{2^n}\leq y\leq 0$); so that $f$ maps $J_n$ 
homeomorphically onto $K_{n+1}\cup J_{n+1}$. The dendrite $X$, as well as the map $f:X\longrightarrow X$, 
are depicted in Figure~\ref{infinitely-branching}.

\begin{figure}[t]
\centering
\includegraphics[scale=0.7]{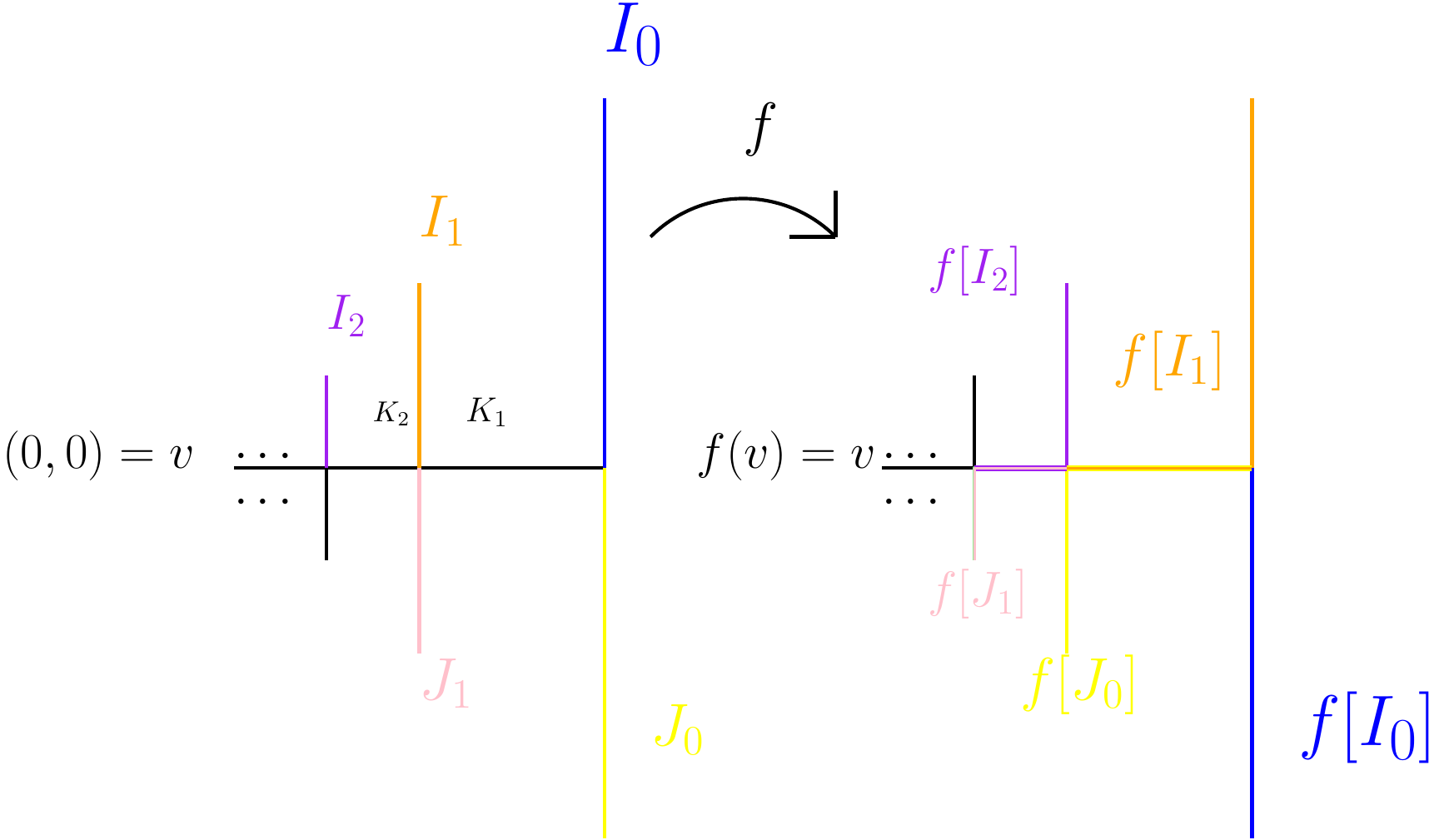}
\caption{The dendrite $X$ that has infinitely many branching points of finite order. The map 
$f \colon X\longrightarrow X$ is not equicontinuous.}
\label{infinitely-branching}
\end{figure}

We will denote by $v$ the point $(0,0)$. Notice that the sequence of endpoints of the $I_n$, 
$\left(\frac{1}{2^n},\frac{1}{2^n}\right)_{n\in\mathbb N}$ (which converges to $v$), along with the increasing 
sequence of indices $(n)_{n\in\mathbb N}$ and $\varepsilon=1$, witness the failure of equicontinuity of $f$ at $v$ (since $f^n\left(\frac{1}{2^n},\frac{1}{2^n}\right)=(1,1)\in I_0$, which is at distance $>1$ from $f^n(v)=v$). So $f$ 
is not equicontinuous.

For every $n\in\mathbb N$ we have 
\begin{equation}\label{p2}
\fix(f^n)=\per(f)=K.
\end{equation}
Thus the sets $\fix(f^n)$, as well as $\per(f)$, are all connected. Using (\ref{p2}) it is straightforward to see that property $(j^\prime)$ of Lemma~\ref{arc-iff-disconnection} is not satisfied. Hence, by the same lemma, $X$ has no $f$-expanding arcs

It remains to show that the function $f^u$ is continuous, whenever $u$ is a nonprincipal ultrafilter. To do this, we define an auxiliary (continuous) function $g \colon X\longrightarrow X$ as follows. First of all, $g\upharpoonright K$ will be the identity map. For every $n\in\mathbb N\cup\{0\}$, we have 
$$
g\left(\frac{1}{2^n},\frac{1}{2^n}\right)=v= g\left(\frac{1}{2^n},-\frac{1}{2^n}\right).
$$ 
Next, if $\left(\frac{1}{2^n},y\right)\in I_n$ is not an endpoint (that is, if $0<y<\frac{1}{2^n}$) then we let 
$m\in\mathbb N\cup\{0\}$ be unique such that 
$$
\frac{1}{2^n}-\frac{1}{2^{n+m}}\leq y<\frac{1}{2^n}-\frac{1}{2^{n+m+1}},
$$ 
\noindent and define
$$
g\left(\frac{1}{2^n},y\right)=\left(\frac{1}{2^{n-m}}+2^{2m+1}\left(y-\frac{1}{2^n}-\frac{1}{2^{n+m}}\right),0\right)
\hspace{.5cm} \mbox{if } m < n,
$$ 
\noindent and 
$$
g\left(\frac{1}{2^n},y\right)=\left(1-2^{2n}\left(y-\left(\frac{1}{2^n}-\frac{1}{2^{2n}}\right)\right),0\right)
  \hspace{.5cm} \mbox{if } n \leq m.
$$ 
Finally, if $\left(\frac{1}{2^n},y\right)\in J_n$ is not an endpoint (i.e., $-\frac{1}{2^n}<y<0$), we let 
$$
g\left(\frac{1}{2^n},y\right)=\left(\frac{1}{2^n}+y,0\right).
$$ 
The function $g$ is continuous and, moreover, for each $x\in X$ we have $\lim_{n\to\infty}f^n(x)=g(x)$ and therefore, for every nonprincipal ultrafilter $u$, it must be the case that $f^u=g$. Thus the function $f^u$ is continuous for 
every nonprincipal ultrafilter $u$.
\end{example}

\subsection{Finite graphs}

The case of finite graphs might be harder to analyze than the case of finite trees. The first difficulty that arises is the fact that the unit circle $\mathbb S^1$ is a finite graph, and there are maps 
$f \colon \mathbb S^1\longrightarrow\mathbb S^1$ (such as, e.g., rotations by an irrational angle), which, though equicontinuous and surjective, lack any periodic 
points. Thus, items (a), (b), (c) and (d) from Theorem~\ref{mainthm} are no longer equivalent if one attempts to replace ``finite tree'' with ``finite graph'' in its statement. For finite graphs with at least one branching point or at 
least one endpoint, however, the equivalence between items (a), (b) and (c) can be established by adapting the argument in the proof of Proposition~\ref{cor:firstfour}. The following example shows that the equivalence between statements (a) and (f) from Theorem~\ref{mainthm} does not hold on finite graphs, even if one demands that the graphs have branching points or endpoints.

\begin{example}\label{finite-graph}
{\em Two finite graphs $X_1,X_2$, and corresponding equicontinuous maps $f_i \colon X_i \longrightarrow X_i$ so that $\fix(f_i)$ is disconnected, for $i\in\{1,2\}$. Here $X_1$ has exactly one branching point and only one endpoint, $X_2$ is a simple closed curve in $X_1$ \emph{(}and thus it has no branching points nor endpoints\emph{)}, and $f_2=f_1\upharpoonright X_2$.} 

The finite graph $X_1$ is given by
\begin{equation*}
X_1=\left\{(x,y)\in\mathbb R^2\big|x \in [1,2]\text{ and }y=0,\text{ or } x\in[-1,1]\text{ and }y=\pm x\right\},
\end{equation*}
and we let $f_1(x,y)=(x,-y)$. Since $f_1^2$ is the identity map, $f_1$ is equicontinuous. However, $\fix(f_1)$ is obviously disconnected. Now 
\begin{equation*}
X_2 = \left\{(x,y)\in\mathbb R^2\big| x\in[-1,1]\text{ and }y=\pm x\right\}
\end{equation*}
and  $f_2=f_1\upharpoonright X_2 \colon X_2 \longrightarrow X_2$ satisfy the requirements.
\end{example}

The observations, along with the example from this subsection suggest that the following might be a worthwhile question (a subset of the following question appears as~\cite[Question~3.10]{ivon-salvador}).

\begin{problem}
Let $(X,f)$ be a discrete dynamical system. Which of the equivalences from \emph{Theorem~\ref{mainthm}} hold 
if we assume that $X$ is an arbitrary finite graph? Which of them hold if we furthermore assume that $X$ has at 
least one branching point or at least one endpoint?
\end{problem}

A subset the next question appears as \cite[Question~3.9]{ivon-salvador}. First recall that
the cone over the harmonic sequence $\{0,1,\frac{1}{2},\frac{1}{3},\ldots\}$ is called the
\emph{harmonic fan}. Attempting to generalize some clauses of Theorem~\ref{mainthm} from finite trees
to non-locally connected continua, we ask the following question.

\begin{problem}
Let $X$ be the harmonic fan and let $f \colon X\longrightarrow X$ be a map. Is it true that
the existence of an $f$-expanding arc in $X$ is equivalent to the fact that $f^u$ is discontinuous for some $u \in 
\mathbb N^*$? Does the existence of an $f$-expanding arc in $X$ implies that there is $m \in \mathbb N$ such that
$X$ contains a $g^m$-divergent sequence?
\end{problem}

We finish with two more questions related to Remark~\ref{rem:aiffd}.

\begin{problem}\label{question:aiffd}
Is there a dendrite $X$ with finitely many branching points and a map $f:X\longrightarrow X$ such that $\per(f)=\bigcap_{n=1}^\infty f^n[X]$ yet $f$ fails to be equicontinuous? An equivalent way of phrasing the same question is: does the implication $(d)\Rightarrow(a)$ of \emph{Theorem~\ref{mainthm}} hold for any dendrite with finitely many branching points?
\end{problem}

A map $f$ as requested in Question~\ref{question:aiffd} would still need to be equicontinuous when restricted to $\bigcap_{n=1}^\infty f^n[X]$ (even though $f$ itself fails to be equicontinuous). Notice that, for $X$ a finite tree and $f:X\longrightarrow X$ a map, equicontinuity of $f$ is equivalent to equicontinuity of the restriction $f\upharpoonright\bigcap_{n=1}^\infty f^n[X]$ by~\cite[Theorem~5.2]{mai}; on the other hand,~\cite[Example 2.9]{sun2014} shows that this equivalence no longer holds if $X$ is an arbitrary dendrite instead. This motivates the following question.

\begin{problem}
Is there a class of dendrites, broader than the class of finite trees, such that if $X$ belongs to the class and $f:X\longrightarrow X$ is a map, then the equicontinuity of $f\upharpoonright\bigcap_{n=1}^\infty f^n[X]$ implies the equicontinuity of $f$?
\end{problem}

\subsection*{Acknowledgements}
The second author was supported by a postdoctoral fellowship from DGAPA-UNAM under the mentoring of the first author. Both authors are grateful to two anonymous referees and one anonymous editor for numerous suggestions that helped to improve the paper.

%%%%%%%%%%% To ease editing, use normal size for the references:

\end{document}